\title{Iteration Complexity of Randomized Block-Coordinate Descent Methods for Minimizing a Composite Function\footnote{The work of the first author was supported in part by EPSRC grant EP/I017127/1 ``Mathematics for vast digital resources''. The second author was supported in part by the Centre for Numerical Algorithms and Intelligent Software (funded by EPSRC grant EP/G036136/1 and the Scottish Funding Council).}}

\author{Peter Richt\'arik\footnote{School of Mathematics, University of Edinburgh, UK, email: peter.richtarik@ed.ac.uk} \qquad \qquad   Martin Tak\'a\v{c}\footnote{School of Mathematics, University of Edinburgh, UK, email: m.takac@sms.ed.ac.uk}\\\\\emph{School of Mathematics}\\\emph{University of Edinburgh}\\\emph{United Kingdom}}

\date{April 2011 (revised on July 4th, 2011)}

\documentclass[dvipsnames,11pt]{article}

\usepackage{amsmath,amssymb, amsthm}

\usepackage{fullpage}

\usepackage{multirow}
\usepackage[cp1250]{inputenc}
\usepackage[T1]{fontenc}
\usepackage{calligra}
\usepackage{amsfonts}
\usepackage{layout}
\usepackage{amsmath}
\usepackage{amsthm}
\usepackage{amssymb}
\usepackage[dvips]{graphicx}
\usepackage{url}
\usepackage{color}
\usepackage{graphicx}
\usepackage{algorithmic}
\usepackage{algorithm}
\usepackage{verbatim}

\newcommand{\nadpis}[1]{\subsubsection*{#1}}

\let\la=\langle
\let\ra=\rangle
\newcommand{\st}{\;:\;}
\newcommand{\ve}[2]{\langle #1 ,  #2 \rangle}
\newcommand{\eqdef}{\stackrel{\text{def}}{=}}

\newcommand{\R}{\mathbf{R}}
\newcommand{\Prob}{\mathbf{P}}
\newcommand{\E}{\mathbf{E}}
\newcommand{\vc}[2]{#1^{(#2)}}

\newcommand{\ncs}[2]{\|#1\|^2_{(#2)}}

\newcommand{\Rw}[2]{\mathcal R_{#1}(#2)}
\newcommand{\Rws}[2]{\mathcal R^2_{#1}(#2)}

\newcommand{\nbp}[2]{\|#1\|_{(#2)}}   
\newcommand{\nbd}[2]{\|#1\|_{(#2)}^*} 
\newcommand{\lf}{\mathcal L}
\newcommand{\U}{U}
\newcommand{\N}{N}

\newcommand{\Lip}{L}

\newcommand{\nnz}[1]{\|#1\|_0}

\DeclareMathOperator{\dom}{dom}         

\DeclareMathOperator{\tr}{tr}           
\DeclareMathOperator{\Diag}{Diag}       

\theoremstyle{plain}
\newtheorem{theorem}{Theorem}
\newtheorem{lemma}[theorem]{Lemma}

\newtheorem{example}{Example}

\theoremstyle{definition}

\begin{document}
\maketitle

\begin{abstract}
In this paper we develop a randomized block-coordinate descent method for minimizing the sum of a smooth and a simple nonsmooth block-separable convex function and prove that it obtains an $\epsilon$-accurate solution with probability at least $1-\rho$ in at most $O(\tfrac{n}{\epsilon} \log \tfrac{1}{\rho})$ iterations, where $n$ is the number of blocks. For strongly convex functions the method converges linearly. This extends recent results of Nesterov [Efficiency of coordinate descent methods on huge-scale optimization problems, CORE Discussion Paper \#2010/2], which cover the smooth case, to composite minimization, while at the same time improving the complexity by the factor of 4 and removing $\epsilon$ from the logarithmic term. More importantly, in contrast with the aforementioned work in which the author achieves the results by applying the method to a regularized version of the objective function with an unknown scaling factor, we show that this is not necessary, thus achieving true iteration complexity bounds. In the smooth case we also allow for arbitrary probability vectors and non-Euclidean norms.  Finally, we demonstrate numerically that the algorithm is able to solve huge-scale $\ell_1$-regularized least squares and support vector machine problems with a billion variables.

\end{abstract}

\section{Introduction}

The goal of this paper, in the broadest sense, is to  develop efficient methods for solving structured convex optimization problems with some or all of these (not necessarily distinct) properties:
\begin{enumerate}
\item \textbf{Size of Data.} The size of the problem, measured as the dimension of the variable of interest, is so large that the computation of a single function value or gradient is prohibitive. There are several situations in which this is the case, let us mention two of them.
\begin{itemize}
    \item \textbf{Memory.} If the dimension of the space of variables is larger than the available memory, the task of forming a gradient or even of evaluating the function value may be impossible to execute and hence the usual gradient methods will not work.
    \item \textbf{Patience.} Even if the memory does not preclude the possibility of taking a gradient step, for large enough problems this step will take considerable time and, in some applications such as image processing,  users might prefer to see/have some intermediary results before a single iteration is over.
\end{itemize}

\item \textbf{Nature of Data.} The nature and structure of data describing the problem may be an obstacle in using current methods for various reasons, including the following.
\begin{itemize}
    \item \textbf{Completeness.} If the data describing the problem is not immediately available in its entirety, but instead arrives incomplete in pieces and blocks over time, with each block ``corresponding to'' one variable, it may not be realistic (for various reasons such as ``memory'' and ``patience'' described above) to wait for the entire data set to arrive before the optimization process is started.
    \item \textbf{Source.} If the data is distributed on a network not all nodes of which are equally responsive or functioning, it may be necessary to work with whatever data is available at a given time.
 \end{itemize}
\end{enumerate}

It appears that a very reasonable approach to solving \emph{some} problems characterized above  is to use \emph{(block) coordinate descent methods} (CD). In the remainder of this section we mix arguments in support of this claim with a brief review of the relevant literature and an outline of our contributions.

\subsection{Block Coordinate Descent Methods}

The basic algorithmic strategy of CD methods is known in the literature under various names such as alternating minimization, coordinate relaxation, linear and non-linear Gauss-Seidel methods, subspace correction and domain decomposition. As working with all the variables of an optimization problem at each iteration may be inconvenient, difficult or impossible for any or all of the reasons mentioned above, the variables are partitioned into manageable blocks, with each iteration focused on updating a single block only, the remaining blocks being fixed. Both for their conceptual and algorithmic simplicity, CD methods were among the first optimization approaches proposed and studied in the literature (see \cite{Bertsekas-Book} and the references therein; for a survey of block CD methods in semidefinite programming we refer the reader to \cite{SemHandbook}). While they seem to have never belonged to the mainstream focus of the optimization community, a renewed interest in CD methods was sparked recently by their successful application in several areas---training support vector machines in machine learning \cite{Lin:2008:DCDM, Lin:2008:CDMLLSVM, ShalevTewari09, Lin:2010:COMS, Lin:IEEEsurvey}, optimization \cite{Tseng:CCMCDM:Smooth, Tseng:CGDM:Nonsmooth, Tseng:CGDMLC:Nonsmooth, Tseng:CBCDM:Nonsmooth, Yun2009, Saha10finite, Nesterov:2010RCDM, Wright:ABCRRO}, compressed sensing \cite{Li:CDOMACSGA}, regression \cite{WuLange:2008}, protein loop closure \cite{Protein2003} and truss topology design \cite{RT:TTD2011}---partly due to a change in the \emph{size} and \emph{nature of data} described above.



\paragraph{Order of coordinates.} Efficiency of a CD method will necessarily depend on the balance between time spent on choosing the block to be updated in the current iteration and the quality of this choice in terms of function value decrease. One extreme possibility is a \emph{greedy} strategy in which the block with the largest descent or guaranteed descent is chosen. In our setup such a strategy is prohibitive as i) it would require all data to be available and ii) the work involved would be excessive due to the size of the problem. Even if one is able to compute all partial derivatives, it seems better to then take a full gradient step instead of a coordinate one, and avoid throwing almost all of the computed information away. On the other end of the spectrum are two very cheap strategies for choosing the incumbent coordinate: \emph{cyclic} and \emph{random}. Surprisingly, it appears that complexity analysis of a cyclic CD method in satisfying generality has not yet been done. The only attempt known to us is the work of Saha and Tewari \cite{Saha10finite}; the authors consider  the case of minimizing a smooth convex function and proceed by establishing a sequence of comparison theorems between the iterates of their method and the iterates of a simple gradient method. Their result requires an isotonicity assumption. Note that a cyclic strategy assumes that the data describing the next block is available when needed which may not always be realistic. The situation with a random strategy seems better; here are some of the reasons:
\begin{itemize}
\item[(i)] Recent efforts  suggest that complexity results are perhaps more readily obtained for randomized methods and that randomization can actually improve the convergence rate \cite{SV:Kaczmarz2009, Leventhal:2008:RMLC, ShalevTewari09}.
\item[(ii)]  Choosing all blocks with equal probabilities should, intuitively, lead to similar results as is the case with a cyclic strategy. In fact, a randomized strategy is able to avoid worst-case  order of coordinates, and hence might be preferable.
\item[(iii)] Randomized choice seems more suitable in cases when not all data is available at all times.
\item[(iv)] One may study the possibility of choosing blocks with different probabilities (we do this in Section~\ref{sec:smooth}). The goal of such a strategy may be either to improve the speed of the method (in Section~\ref{sec:exp:lasso} we introduce a speedup heuristic based on adaptively changing the probabilities), or a more realistic modeling of the availability frequencies of the data defining each block.
\end{itemize}


\paragraph{Step size.} 
Once a coordinate (or a block of coordinates) is chosen to be updated in the current  iteration, partial derivative can be used to drive the steplength in the same way as it is done in the usual gradient methods. As it is sometimes the case that the computation of a partial derivative is \emph{much cheaper and less memory demanding} than the computation of the entire gradient, CD methods seem to be promising candidates for problems described above. It is important that line search, if any is implemented, is very efficient. The entire data set is either huge or not available and hence it is not reasonable to use function values at any point in the algorithm, including the line search. Instead, cheap partial derivative and other information derived from the  problem structure should be used to drive such a method.


\subsection{Problem Description and Our Contribution}

\paragraph{The problem.}
In this paper we study the \emph{iteration complexity} of simple randomized block coordinate decent methods applied to the problem of minimizing a \emph{composite objective function}, i.e., a function formed as the sum of a smooth convex and a simple nonsmooth convex term:
\begin{equation}\label{eq:P}\min_{x\in \R^\N} F(x) \eqdef f(x) + \Psi(x).\end{equation}
We assume that this problem has a minimum ($F^*>-\infty$),  $f$  has (block) coordinate Lipschitz gradient, and $\Psi$ is a (block) separable proper closed convex extended real valued function (these properties will be defined precisely in Section~\ref{sec:prelim}). Possible choices of $\Psi$ include:

%
%

\begin{enumerate}
 \item[(i)] $\Psi \equiv 0$. This covers the case of \emph{smooth minimization}. Complexity results are given in \cite{Nesterov:2010RCDM}.
 \item[(ii)] $\Psi$ is the indicator function of a block-separable convex set (such as a box). This choice models \emph{problems with constraints on blocks of variables}; iteration complexity results are given in \cite{Nesterov:2010RCDM}.
 \item[(iii)] $\Psi(x) \equiv \lambda \|x\|_1$ for $\lambda>0$. In this case we can decompose $\R^\N$ onto $\N$ blocks. Increasing $\lambda$ encourages  the solution of \eqref{eq:P} to be sparser \cite{Wright:SRSA}. Applications abound in, for instance, machine learning \cite{Lin:2008:CDMLLSVM}, statistics \cite{Tibshirani:1996} and signal processing \cite{Li:CDOMACSGA}.
 \item[(iv)] There are many more choices such as the elastic net \cite{ZouHastie:elastic-net:2005}, group lasso \cite{YuanLin:GroupLasso, Meier:GroupLasso,QinScheinbergGoldfarb:GroupLasso} and sparse group lasso \cite{SparseGroupLasso2010}.
\end{enumerate}

\paragraph{Iteration complexity results.} Strohmer and Vershynin \cite{SV:Kaczmarz2009} have recently proposed a randomized Karczmarz method for solving overdetermined consistent systems of linear equations and proved that the method enjoys global linear convergence whose rate can be expressed in terms of the condition number of the underlying matrix. The authors claim that for certain problems their approach can be more efficient than the conjugate gradient method. Motivated by these results, Leventhal and Lewis \cite{Leventhal:2008:RMLC} studied the problem of solving a system of linear equations and inequalities and in the process gave iteration complexity bounds for a randomized CD method applied to the problem of minimizing a convex quadratic function. In their method the probability of choice of each coordinate is proportional to the corresponding diagonal element of the underlying positive semidefinite matrix defining the objective function. These diagonal elements can be interpreted as Lipschitz constants of the derivative of a restriction of the quadratic objective onto one-dimensional lines parallel to the coordinate axes. In the general (as opposed to quadratic) case considered in this paper \eqref{eq:P}, these Lipschitz constants will play an important role as well. Lin et al.~\cite{Lin:2008:CDMLLSVM} derived iteration complexity results for several smooth objective functions appearing in machine learning. Shalev-Schwarz and Tewari \cite{ShalevTewari09} proposed a randomized coordinate descent method with uniform probabilities for minimizing $\ell_1$-regularized smooth convex problems. They first transform the problem into a box constrained smooth problem by doubling the dimension and then apply a coordinate gradient descent method in which each coordinate is chosen with equal probability. Nesterov \cite{Nesterov:2010RCDM} has recently analyzed randomized coordinate descent methods in the smooth unconstrained and  box-constrained setting, in effect extending and improving upon some of the results in \cite{Leventhal:2008:RMLC, Lin:2008:CDMLLSVM, ShalevTewari09} in several ways.

While the \emph{asymptotic convergence rates} of some variants of CD methods are well understood \cite{Tseng:CCMCDM:Smooth, Tseng:CGDM:Nonsmooth, Tseng:CGDMLC:Nonsmooth, Tseng:CBCDM:Nonsmooth, Yun2009}, \emph{iteration complexity} results are very rare.
To the best of our knowledge, randomized CD algorithms for minimizing a composite function have been proposed and analyzed (in the iteration complexity sense) in a few special cases only: a) the  unconstrained convex quadratic case \cite{Leventhal:2008:RMLC}, b) the smooth unconstrained ($\Psi\equiv 0$) and the smooth block-constrained case ($\Psi$ is the indicator function of a direct sum of boxes)  \cite{Nesterov:2010RCDM} and c) the $\ell_1$-regularized case \cite{ShalevTewari09}. As the approach in \cite{ShalevTewari09} is to rewrite the problem into a smooth box-constrained format first, the results of  \cite{Nesterov:2010RCDM} can be viewed as a (major) generalization and improvement of those in \cite{ShalevTewari09} (the results were obtained independently).

\paragraph{Contribution.} In this paper we further improve upon and extend and simplify the iteration complexity results of Nesterov \cite{Nesterov:2010RCDM}, treating the problem of minimizing the sum of a smooth convex and a simple nonsmooth convex block separable function \eqref{eq:P}. We focus exclusively on simple (as opposed to accelerated) methods. The reason for this is that the per-iteration work of the accelerated algorithm in \cite{Nesterov:2010RCDM} on huge scale instances of problems with \emph{sparse} data (such as the Google problem where sparsity corresponds to each website linking only to a few other websites or the sparse problems we consider in Section~\ref{sec:experiments}) is excessive. In fact, even the author does not recommend using the accelerated method for solving such problems; the simple methods seem to be more efficient.

Each algorithm of this paper is supported by a high probability iteration complexity result. That is, for any given \emph{confidence level} $0<\rho<1$ and \emph{error tolerance} $\epsilon>0$, we give an explicit expression for the number of iterations $k$ which guarantee that the method produces a random iterate $x_k$ for which \[\Prob(F(x_k)-F^*\leq \epsilon) \geq 1-\rho.\]

Table~\ref{T:0} summarizes the main complexity results of this paper. Algorithm~\ref{algorithm:UCD}---Uniform (block) Coordinate Descent for Composite functions (UCDC)---is a method where at each iteration the block of coordinates to be updated (out of a total of $n\leq \N$ blocks) is chosen uniformly at random. Algorithm~\ref{algorithm:RCDM-smooth}---Randomized (block) Coordinate Descent for Smooth functions (RCDS)---is a method where at each iteration block $i\in\{1,\dots,n\}$ is chosen with probability $p_i$. Both of these methods are special cases of the generic Algorithm~\ref{algorithm:NRCDM}; Randomized (block) Coordinate Descent for Composite functions (RCDC).

\begin{table}[!ht]
\begin{center}
{\footnotesize
\begin{tabular}{|ccc|}
  \hline
  Algorithm  &  Objective & Complexity\\
  \hline
  \hline
 &&\\
\begin{tabular}{c}
   Algorithm~\ref{algorithm:UCD} (UCDC) \\
   (Theorem~\ref{thm:composite_general_f})
   \end{tabular}
   &
   \begin{tabular}{c}
     convex \\
     composite
   \end{tabular}
   &
   \begin{tabular}{c}
                                                       $\tfrac{2n\max\{\Rws{\Lip}{x_0}, F(x_0)-F^*\}}{\epsilon}(1+\log \tfrac{1}{\rho})$ \\
                                                       $\tfrac{2n\Rws{\Lip}{x_0}}{\epsilon}\log \left(\tfrac{F(x_0)-F^*}{\epsilon \rho}\right)$
   \end{tabular}\\
   &&\\
  \begin{tabular}{c}
  Algorithm~\ref{algorithm:UCD} (UCDC) \\
    (Theorem~\ref{thm:nonsmooth:stornglyconvex:highprobresult}) \\
  \end{tabular}
   &
   \begin{tabular}{c}
     strongly convex \\
     composite \\
   \end{tabular}
   &
   $\max\{\tfrac{4}{\mu},\tfrac{\mu}{\mu-1}\} n \log\left(\tfrac{F(x_0)-F^*}{\rho\epsilon}\right)$\\
   &&\\
  \begin{tabular}{c}
  Algorithm~\ref{algorithm:RCDM-smooth} (RCDS) \\
    (Theorem~\ref{thm:smooth_main})
  \end{tabular}
   &
   \begin{tabular}{c}
     convex \\
     smooth \\
   \end{tabular}
   &
   $\tfrac{2\Rws{LP^{-1}}{x_0}}{\epsilon} (1 + \log \tfrac{1}{\rho}) -2$\\
   &&\\
  \begin{tabular}{c}
  Algorithm~\ref{algorithm:RCDM-smooth} (RCDS) \\
    (Theorem~\ref{thm:smooth_main_strong})
  \end{tabular}
   &
   \begin{tabular}{c}
     strongly convex \\
     smooth \\
   \end{tabular}
   &
   $\tfrac{1}{\mu}\log \left(\tfrac{f(x_0)-f^*}{\epsilon\rho}\right)$\\
   &&\\
  \hline
\end{tabular}
}
\end{center}
\caption{Summary of complexity results obtained in this paper.}
\label{T:0}
\end{table}

The symbols $P, \Lip, \Rws{W}{x_0}$ and $\mu$ appearing in Table~\ref{T:0} will be defined precisely in further sections. For now it suffices to say that $\Lip$ encodes the (block) coordinate Lipschitz constants of the gradient of $f$, $P$ encodes the probabilities $\{p_i\}$, $\Rws{W}{x_0}$ is a measure of distance of the initial iterate $x_0$ from the set of minimizers of the problem \eqref{eq:P} in a norm defined by $W$ (see Section~\ref{sec:prelim}) and $\mu$ is the strong convexity parameter of $F$ (see Section~\ref{sec:composite_strong}). In the nonsmooth case $\mu$ depends on $\Lip$ and the smooth case it depends both on $\Lip$ and $P$.

Let us now briefly outline the main similarities and differences between our results and those in  \cite{Nesterov:2010RCDM}. A more detailed and expanded discussion can be found in Section~\ref{sec:comparison}.



\begin{enumerate}
\item \textbf{Composite setting.} We consider the composite setting \eqref{eq:P}, whereas \cite{Nesterov:2010RCDM} covers the unconstrained and constrained smooth setting only.
\item \textbf{No need for regularization.} Nesterov's high probability results in the case of minimizing a function which is not strongly convex are based on regularizing the objective to make it strongly convex and then running the method on the regularized function. Our contribution here is that we show that no regularization is needed by doing a more detailed analysis using a thresholding argument (Theorem~\ref{l:randomVariableTrick}).
\item \textbf{Better complexity.} Our complexity results are better by the constant factor of 4. Also, we have removed $\epsilon$ from under the logarithm.
\item \textbf{General probabilities.} Nesterov considers probabilities $p_i$ proportional to $\Lip_i^{\alpha}$, where $\alpha\geq 0$ is a parameter. High probability results are proved in \cite{Nesterov:2010RCDM} for $\alpha \in \{0,1\}$ only. Our results in the smooth case hold for an arbitrary probability vector $p$.
\item \textbf{General norms.} Nesterov's expectation results (Theorems 1 and 2) are proved for general norms. 
However, his high probability results are proved for Euclidean norms only. In our approach all results hold for general norms.
\item \textbf{Simplification.} Our analysis is more compact. 
\end{enumerate}

In the numerical experiments  section we focus on \emph{sparse} $\ell_1$-regularized regression and support vector machine problems. For these problems we introduce a powerful \emph{speedup heuristic} based on adaptively changing the probability vector throughout the iterations (Section~\ref{sec:exp:lasso}; ``speedup by shrinking'').

\paragraph{Contents.} This paper is organized as follows. We start in Section~\ref{sec:prelim} by defining basic notation, describing the block structure of the problem, stating assumptions and describing the generic randomized block-coordinate descent algorithm (RCDC). In Section~\ref{sec:composite} we study the performance of a uniform variant (UCDC) of RCDC as applied to a composite objective function and in Section~\ref{sec:smooth} we analyze a smooth variant (RCDS) of RCDC; that is, we study the performance of RCDC on a smooth objective function.
In Section~\ref{sec:comparison} we compare known complexity results for CD methods with the ones established in this paper. Finally, in Section~\ref{sec:experiments} we demonstrate the efficiency of the method on $\ell_1$-regularized sparse regression and linear support vector machine problems.

\section{Assumptions and the Algorithm} \label{sec:prelim}

\paragraph{Block structure.} We  model the block structure of the problem by decomposing the space $\R^\N$ into $n$ subspaces as follows. Let $\U\in \R^{\N\times \N}$ be a column permutation of the $\N\times \N$ identity matrix and further let $\U = [\U_1,\U_2,\dots,\U_n]$ be a decomposition of $\U$ into $n$ submatrices, with $\U_i$ being of size $\N \times \N_i$,
where $\sum_i \N_i = \N$. Clearly, any vector $x\in \R^\N$ can be written uniquely as $x = \sum_i \U_i x^{(i)}$, where $x^{(i)}=\U_i^T x \in \R_i \equiv \R^{\N_i}$. Also note that
\begin{equation}\label{eq:U_iU_j} \U_i^T \U_j =
\begin{cases} \N_i\times \N_i \quad \text{identity matrix,} & \text{ if } i=j,\\
 \N_i\times \N_j \quad \text{zero matrix,}&  \text{ otherwise.}
 \end{cases}
\end{equation}
For simplicity we will write $x = (x^{(1)},\dots,x^{(n)})^T$. We equip $\R_i$ with a pair of conjugate Euclidean norms:
\begin{equation}\label{eq:blocknorm}\|t\|_{(i)} = \ve{B_i t}{t}^{1/2}, \qquad \nbd{t}{i} = \ve{B_i^{-1} t}{t}^{1/2}, \qquad t\in \R_i,\end{equation}
where $B_i\in \R^{\N_i\times \N_i}$ is a positive definite matrix and $\ve{\cdot}{\cdot}$ is the standard Euclidean inner product.

\begin{example} Let $n=\N$, $\N_i=1$ for all $i$ and $U = [e_1,e_2,\dots,e_n]$ be the $n\times n$ identity matrix. Then $U_i=e_i$ is the $i$-th unit vector and
$x^{(i)} = e_i^Tx\in \R_i =\R$ is the $i$-th coordinate of $x$. Also, $x = \sum_i e_ix^{(i)}$. If we let $B_i=1$ for all $i$, then $\|t\|_{(i)} = \|t\|^*_{(i)} = |t|$ for all $t\in \R$.
\end{example}

\paragraph{Smoothness of $f$.} We assume throughout the paper that the gradient of $f$ is block coordinate-wise Lipschitz, uniformly in $x$, with positive constants $\Lip_1,\dots,\Lip_n$, i.e., that for all $x\in \R^\N$, $t\in \R_i$ and $i$ we have
\begin{equation}\label{eq:f_iLipschitzder}
 \nbd{\nabla_i f(x+\U_i t)-\nabla_i f(x)}{i} \leq \Lip_i \nbp{t}{i},
  \end{equation}
where \begin{equation}\label{eq:nabla_if(x)}
\nabla_i f(x)  \eqdef (\nabla f(x))^{(i)} = \U^T_i \nabla f(x) \in \R_i.\end{equation} An important consequence of \eqref{eq:f_iLipschitzder} is the following standard inequality \cite{NesterovBook}:
\begin{equation}\label{eq:Lipschitz_ineq}f(x+\U_i t) \leq f(x) + \ve{\nabla_i f(x)}{t} + \tfrac{\Lip_i}{2}\nbp{t}{i}^2.\end{equation}

\paragraph{Separability of $\Psi$.} We assume that $\Psi$ is block separable, i.e., that it can be decomposed as follows:
\begin{equation}\label{eq:Psi_block_def}
  \Psi(x)=\sum_{i=1}^n \Psi_i(x^{(i)}),
\end{equation}
where the functions $\Psi_i:\R_i\to \R$ are convex and closed.

\paragraph{The algorithm.}

Notice that an upper bound on $F(x+\U_i t)$, viewed as a function of $t\in \R_i$, is readily available:
\begin{align}F(x+\U_i t) &\stackrel{\eqref{eq:P}}{=} f(x+\U_i t) + \Psi(x+\U_i t)\stackrel{\eqref{eq:Lipschitz_ineq}}{\leq} f(x) + V_i(x,t) + C_i(x), \label{eq:upper_bound}\end{align}
where
\begin{equation}\label{eq:V}V_i(x,t) \eqdef \ve{\nabla_i f(x)}{t} + \tfrac{\Lip_i}{2}\nbp{t}{i}^2 + \Psi_i(x^{(i)} + t)\end{equation}
and \begin{equation}\label{eq:Ci}C_i(x) \eqdef \sum_{j\neq i}\Psi_j(x^{(j)}).\end{equation}

We are now ready to describe the generic method. Given iterate $x_k$, Algorithm~\ref{algorithm:NRCDM} picks block $i_k=i\in \{1,2,\dots,n\}$ with probability $p_i>0$  and then updates the $i$-th block of $x_k$ so as to minimize (exactly) in $t$ the upper bound \eqref{eq:upper_bound} on $F(x_k +\U_{i} t)$. Note that in certain cases it is possible to minimize $F(x_k +\U_{i} t)$ directly; perhaps in a closed form. This is the case, for example, when $f$ is a convex quadratic.

\begin{algorithm}[h!]
\caption{RCDC$(p,x_0)$ (\textbf{R}andomized \textbf{C}oordinate \textbf{D}escent for \textbf{C}omposite Functions)}
\begin{algorithmic} \label{algorithm:NRCDM}
\FOR {$k=0,1,2,\dots$}
 \STATE Choose $i_{k} = i \in \{1,2,\dots,n\}$ with probability $p_i$
 \STATE
 $\displaystyle T^{(i)}(x_k) \;\eqdef\; \arg \min \{V_{i}(x_k,t)\;:\; t\in\R_i\}$
 \STATE $x_{k+1} = x_k + \U_{i}   T^{(i)}(x_k)$

\ENDFOR
\end{algorithmic}
\end{algorithm}

The iterates $\{x_k\}$ are random vectors and the values $\{F(x_k)\}$ are random variables. Clearly, $x_{k+1}$ depends only on $x_k$.
As our analysis will be based on the (expected) per-iteration  decrease of the objective function, the results will hold even if we replace $V_i(x_k,t)$ by $F(x_k +\U_{i} t)$ in Algorithm~\ref{algorithm:NRCDM}.

\paragraph{Global structure.} For fixed positive scalars $w_1,\dots,w_n$ let $W=\Diag (w_1,\dots,w_n)$ and define a pair of conjugate norms in $\R^\N$ by
\begin{equation}\label{eq:normW}\|x\|_W = \left[\sum_{i=1}^n w_i \ncs{\vc{x}{i}}{i}\right]^{1/2},\end{equation}
\begin{equation}\label{eq:dual_norm}  \|y\|_W^* = \max_{\|x\|_W\leq 1} \ve{y}{x} = \left[\sum_{i=1}^n w_i^{-1} ( \nbd{y^{(i)}}{i})^2\right]^{1/2}.\end{equation}
In the the subsequent analysis we will use $W=L$ (Section~\ref{sec:composite}) and $W = LP^{-1}$ (Section~\ref{sec:smooth}), where $L=\Diag(\Lip_1,\dots,\Lip_n)$ and $P=\Diag(p_1,\dots,p_n)$.

The set of optimal solutions of \eqref{eq:P} is denoted by $X^*$
and $x^*$ is any element of that set. Define
\[\Rw{W}{x} = \max_y \max_{x^*\in X^*} \{\|y-x^*\|_W \;:\; F(y) \leq F(x)\},\]
which is a measure of the size of the level set of $F$ given by $x$. In most of the results in this paper we will need to assume that $\Rw{W}{x_0}$ is finite for the initial iterate $x_0$ and $W=L$ or $W=LP^{-1}$.

\paragraph{A technical result.} The next simple result is the main technical tool enabling us to simplify and improve the corresponding analysis in \cite{Nesterov:2010RCDM}. It will be used with $\xi_k = F(x_k)-F^*$.

\begin{theorem}\label{l:randomVariableTrick} Let $\xi_0>0$ be a constant, $0<\epsilon<\xi_0$, and
consider a nonnegative nonincreasing sequence of (discrete) random variables $\{\xi_k\}_{k\geq 0}$ with one of the following properties:
\begin{enumerate}
\item[(i)] $\E[\xi_{k+1} \;|\; {\xi_k}] \leq \xi_k - \tfrac{\xi_k^2}{c}$, for all $k$, where $c>0$ is a constant,
\item[(ii)] $\E[\xi_{k+1} \;|\; {\xi_k}] \leq (1-\tfrac{1}{c}) \xi_k$, for all $k$ such that $\xi_k\geq \epsilon$, where $c>1$ is  a constant.
\end{enumerate}
Choose confidence level $\rho \in (0,1)$. If property (i) holds and we choose $\epsilon < c$ and
\begin{equation}\label{eq:MainTrick:k}K \geq \tfrac{c}{\epsilon} (1 + \log \tfrac{1}{\rho}) + 2   - \tfrac{c}{\xi_0},\end{equation}
or if property (ii) holds, and we choose
\begin{equation}\label{eq:MainTrick:k2}K\geq c \log \tfrac{\xi_0}{ \epsilon \rho},\end{equation}
then
\begin{equation}\label{eq:MainTrick:prob} \Prob(\xi_K \leq \epsilon) \geq 1-\rho.
\end{equation}
\end{theorem}

\begin{proof} Notice that the sequence $\{\xi_k^\epsilon\}_{k\geq 0}$ defined by
\[\xi_k^\epsilon = \begin{cases}\xi_k & \text{if } \xi_k\geq \epsilon,\\
0 & \text{otherwise,}\end{cases}\]
satisfies
\begin{align}\label{eq:MT0}\xi_{k}^\epsilon \leq \epsilon \quad &\Leftrightarrow \quad \xi_k \leq \epsilon, \qquad k\geq 0.
\end{align}
Therefore, by Markov inequality,
\[\Prob(\xi_k > \epsilon) =  \Prob(\xi_k^\epsilon > \epsilon) \leq \tfrac{\E[\xi_k^\epsilon]}{\epsilon},\] and hence it suffices to show that
\begin{equation}\label{eq:theta_bound}\theta_K \leq \epsilon \rho,\end{equation}
where $\theta_k \eqdef \E[\xi_k^\epsilon]$. If property (i) holds, then
\begin{equation}
\label{eq:MT1} \E[\xi^\epsilon_{k+1} \;|\; {\xi^\epsilon_k}] \leq \xi^\epsilon_k - \tfrac{(\xi^\epsilon_k)^2}{c}, \qquad \E[\xi^\epsilon_{k+1} \;|\; {\xi^\epsilon_k}] \leq (1-\tfrac{\epsilon}{c})\xi^\epsilon_k, \qquad k \geq  0,
\end{equation}
and by taking expectations (using convexity of $t\mapsto t^2$ in the first case) we obtain
\begin{eqnarray}\label{eq:MT2a}\theta_{k+1} & \leq & \theta_k - \tfrac{\theta_k^2}{c}, \qquad k\geq 0,\\
\label{eq:MT2b}\theta_{k+1}& \leq & (1-\tfrac{\epsilon}{c})\theta_k, \qquad k\geq 0.\end{eqnarray}
Notice that \eqref{eq:MT2a} is better than \eqref{eq:MT2b} precisely when $\theta_k>\epsilon$. Since
\[\tfrac{1}{\theta_{k+1}} - \tfrac{1}{\theta_k} = \tfrac{\theta_k-\theta_{k+1}}{\theta_{k+1}\theta_k} \geq \tfrac{\theta_k-\theta_{k+1}}{\theta_k^2} \stackrel{\eqref{eq:MT2a}}{\geq} \tfrac{1}{c},\]
we have $\tfrac{1}{\theta_{k}} \geq \tfrac{1}{\theta_0} + \tfrac{k}{c} = \tfrac{1}{\xi_0} + \tfrac{k}{c}$. Therefore,  if we let $k_1\geq \tfrac{c}{\epsilon}  - \tfrac{c}{\xi_0}$, we obtain $\theta_{k_1}\leq \epsilon$. Finally, letting $k_2 \geq \tfrac{c}{\epsilon}\log \tfrac{1}{\rho}$, we have
\[\theta_K \stackrel{\eqref{eq:MainTrick:k}}{\leq} \theta_{k_1+k_2} \stackrel{\eqref{eq:MT2b}}{\leq} (1-\tfrac{\epsilon}{c})^{k_2}\theta_{k_1}\leq ((1-\tfrac{\epsilon}{c})^{\tfrac{1}{\epsilon}})^{c\log \tfrac{1}{\rho}} \epsilon \leq (e^{-\frac{1}{c}})^{c\log \tfrac{1}{\rho}}\epsilon = \epsilon \rho,\]
establishing \eqref{eq:theta_bound}. 
If property (ii) holds, then $\E[\xi_{k+1}^\epsilon \;|\; \xi_k^\epsilon] \leq (1-\tfrac{1}{c})\xi_k^\epsilon$ for all $k$, and hence
\[\theta_K \leq (1-\tfrac{1}{c})^K \theta_0 = (1-\tfrac{1}{c})^K \xi_0 \stackrel{\eqref{eq:MainTrick:k2}}{\leq} ((1-\tfrac{1}{c})^{c})^{\log \tfrac{\xi_0}{\epsilon\rho}} \xi_0 \leq (e^{-1})^{\log \tfrac{\xi_0}{\epsilon\rho}} \xi_0 = \epsilon\rho,\]
again establishing \eqref{eq:theta_bound}.
\end{proof}

\paragraph{Restarting.} Note that similar, albeit \emph{slightly weaker,} high probability results can be achieved by \emph{restarting} as follows. We run the random process $\{\xi_k\}$ repeatedly $r=\lceil \log \tfrac{1}{\rho}\rceil$ times, always starting from $\xi_0$, each time for the same number of iterations $k_1$  for which $\Prob(\xi_{k_1}>\epsilon) \leq \tfrac{1}{e}$. It then follows that the probability that all $r$ values $\xi_{k_1}$ will be larger than $\epsilon$ is at most $(\tfrac{1}{e})^r \leq \rho$. Note that the restarting technique demands that we perform $r$ evaluations of the objective function; this is not needed in the one-shot approach covered by the theorem.

It remains to estimate $k_1$ in the two cases of Theorem~\ref{l:randomVariableTrick}. We argue that in case (i) we can choose $k_1 = \lceil\tfrac{c}{\epsilon/e}-\tfrac{c}{\xi_0}\rceil$. Indeed, using similar arguments as in Theorem~\ref{l:randomVariableTrick} this leads to $\E[\xi_{k_1}]\leq \tfrac{\epsilon}{e}$, which by Markov inequality implies that in a single run of the process we have \[\Prob(\xi_{k_1} > \epsilon) \leq \tfrac{\E[\xi_{k_1}]}{\epsilon} \leq  \tfrac{\epsilon/e}{\epsilon} = \tfrac{1}{e}.\] Therefore,
\[K = \lceil\tfrac{ec}{\epsilon}-\tfrac{c}{\xi_0}\rceil \lceil\log \tfrac{1}{\rho}\rceil\]
iterations suffice  in case (i). A similar restarting technique can be applied in case (ii).

\paragraph{Tightness.} It can be shown on simple examples that the bounds in the above result are \emph{tight}. 
\section{Coordinate Descent for Composite Functions} \label{sec:composite}

In this section we study the performance of Algorithm~\ref{algorithm:NRCDM} in the special case when all probabilities are chosen to be the same, i.e., $p_i=\tfrac{1}{n}$ for all $i$. For easier future reference we set this method apart and give it a name (Algorithm~\ref{algorithm:UCD}).

\begin{algorithm}[ht!]
\caption{UCDC$(x_0)$ (\textbf{U}niform \textbf{C}oordinate \textbf{D}escent for \textbf{C}omposite Functions)}
\begin{algorithmic} \label{algorithm:UCD}
\FOR {$k=0,1,2,\dots$}
 \STATE Choose $i_{k} = i \in \{1,2,\dots,n\}$ with probability $\tfrac{1}{n}$
 \STATE
 $\displaystyle T^{(i)}(x_k) = \arg \min \{V_{i}(x_k,t)\;:\; t\in\R_{i}\}$
 \STATE $x_{k+1} = x_k + \U_{i}  T^{(i)}(x_k)$
\ENDFOR
\end{algorithmic}
\end{algorithm}
The following function plays a central role in our analysis:
\begin{equation}\label{eq:H(x,t)}
H(x,T) \eqdef f(x) + \ve{\nabla f(x)}{T} + \tfrac{1}{2} \|T\|_L^2 +\Psi(x+T).
\end{equation}
Comparing \eqref{eq:H(x,t)} with \eqref{eq:V} using \eqref{eq:U_iU_j}, \eqref{eq:nabla_if(x)}, \eqref{eq:Psi_block_def} and \eqref{eq:normW} we get
\begin{equation}\label{eq:H_and_V}H(x,T) = f(x) + \sum_{i=1}^n V_i(x,T^{(i)}).\end{equation}
Therefore, the vector $T(x) = (T^{(1)}(x),\dots,T^{(n)}(x))$, with the components $T^{(i)}(x)$  defined in Algorithm~\ref{algorithm:NRCDM}, is the minimizer of $H(x,\cdot)$:

\begin{equation}\label{eq:T(x)def}T(x) = \arg\min_{T\in\R^\N} H(x,T).\end{equation}

Let us start by establishing an auxiliary result which will be used repeatedly.

\begin{lemma} \label{lem:op1} Let $\{x_k\}, \; k\geq 0$, be the random iterates generated by UCDC$(x_0)$. Then
  \begin{equation}\label{eq:op1}
  \E[F(x_{k+1})-F^* \;|\; x_k]  \leq  \tfrac{1}{n}\; (H(x_k,T(x_k))-F^*) + \tfrac{n-1}{n} \;(F(x_k)-F^*).
 \end{equation}
\end{lemma}
\begin{proof}
\begin{eqnarray*}
\E[F(x_{k+1}) \;|\; x_k]
  &=& \sum_{i=1}^n \tfrac{1}{n} F(x_k+\U_i T^{(i)}(x_k))\\
& \stackrel{\eqref{eq:upper_bound}}{\leq}& \tfrac{1}{n}\sum_{i=1}^n  [f(x_k) + V_i(x_k,T^{(i)}(x_k)) + C_i(x_k)]\\
& \stackrel{\eqref{eq:H_and_V}}{=}& \tfrac{1}{n}H(x_k,T(x_k)) + \tfrac{n-1}{n}f(x_k) + \tfrac{1}{n}\sum_{i=1}^n C_i(x_k)\\
& \stackrel{\eqref{eq:Ci}}{=}& \tfrac{1}{n}H(x_k,T(x_k)) + \tfrac{n-1}{n}f(x_k) + \tfrac{1}{n}\sum_{i=1}^n \sum_{j\neq i}
\Psi_j(x_k^{(j)})\\
&= & \tfrac{1}{n}H(x_k,T(x_k)) + \tfrac{n-1}{n}F(x_k).
\end{eqnarray*}
\end{proof}

%
%

\subsection{Convex Objective}

In order for Lemma~\ref{lem:op1} to be useful, we need to estimate $H(x_k,T(x_k))-F^*$ from above in terms of $F(x_k)-F^*$.

\begin{lemma}\label{lem:H-Fstar} Fix $x^*\in X^*$, $x\in \dom \Psi$ and let $R = \|x-x^*\|_L$. Then
\begin{equation}\label{eq:nonsmooth:gamma1}
 H(x,T(x)) - F^* \leq \begin{cases} \left(1-\tfrac{F(x)-F^*}{2R^2}\right)(F(x)-F^*), \quad & \text{if } F(x)-F^*\leq R^2,\\
\tfrac{1}{2} R^2 < \tfrac{1}{2}(F(x)-F^*), \quad & \text{otherwise.}
\end{cases}
\end{equation}
\end{lemma}

\begin{proof}
\begin{eqnarray}
\notag H(x,T(x)) &\stackrel{\eqref{eq:T(x)def}}{=}& \min_{T\in \R^{\N}} H(x,T)\\
\notag & = & \min_{y\in\R^{\N}} H(x,y-x)\\
\notag & \stackrel{\eqref{eq:H(x,t)}}{\leq}&  \min_{y\in \R^{\N}} f(x)+ \ve{\nabla f(x)}{y-x} + \Psi(y)+\tfrac{1}{2} \|y-x\|_L^2\\
\notag & \leq& \min_{y\in \R^{\N}}   F(y) + \tfrac{1}{2} \|y-x\|_L^2\\
\notag &\leq & \min_{\alpha \in [0,1]}   F(\alpha x^* + (1-\alpha)x) + \tfrac{\alpha^2}{2} \|x-x^*\|_L^2\\
\label{eq:R2} &\leq & \min_{\alpha \in [0,1]} F(x)-\alpha (F(x)-F^*)+ \tfrac{\alpha^2}{2} R^2.
\end{eqnarray}
Minimizing \eqref{eq:R2} in $\alpha$ gives $\alpha^* = \min\left\{1,(F(x)-F^*)/R^2\right\}$; the result follows.
\end{proof}

We are now ready to estimate the number of iterations needed to push the objective value within $\epsilon$ of the optimal value with high probability. Note that since $\rho$ appears under the logarithm and hence it is easy to attain high confidence.

\begin{theorem}\label{thm:composite_general_f}
Choose initial point $x_0$ and target confidence $0<\rho<1$. Further, let the target accuracy $\epsilon>0$ and iteration counter $k$ be chosen in any of the following two ways:
\begin{enumerate}
\item[(i)] $\epsilon<F(x_0)-F^*$  and
\begin{equation}\label{eq:k_composite_nonstrong} k \geq \tfrac{2n \max\{\Rws{L}{x_0}, F(x_0)-F^*\}}{\epsilon} \left(1 + \log \tfrac{1}{\rho}\right) + 2 - \tfrac{2n\max\{\Rws{L}{x_0}, F(x_0)-F^*\}}{F(x_0)-F^*},\end{equation}
\item[(ii)] $\epsilon < \min\{\Rws{L}{x_0}, F(x_0)-F^*\}$ and
\begin{equation}\label{eq:k_composite_nonstrong2} k \geq \tfrac{2n \Rws{L}{x_0}}{\epsilon} \log \tfrac{F(x_0)-F^*}{\epsilon\rho }.\end{equation}
\end{enumerate}
If $x_k$ is the random point generated by UCDC$(x_0)$ as applied to the convex function $F$, then
\[\Prob(F(x_k)-F^*\leq \epsilon) \geq 1-\rho.\]
\end{theorem}
\begin{proof}Since $F(x_k)\leq F(x_0)$ for all $k$, we have $\|x_k-x^*\|_L\leq \Rw{L}{x_0}$ for all $x^*\in X^*$. Lemma~\ref{lem:op1} together with Lemma~\ref{lem:H-Fstar} then imply that the following holds for all $k$:
\begin{eqnarray}\E[F(x_{k+1}) - F^* \;|\; x_k]
&\leq& \tfrac{1}{n}\max\left\{1-\tfrac{F(x_k)-F^*}{2\|x_k-x^*\|_L^2},\tfrac{1}{2} \right\}(F(x_k)-F^*) + \tfrac{n-1}{n}(F(x_k)-F^*)\notag\\
&=&    \max\left\{1-\tfrac{F(x_k)-F^*}{2n\|x_k-x^*\|_L^2},1-\tfrac{1}{2n} \right\}(F(x_k)-F^*)\notag\\
&\leq& \max\left\{1-\tfrac{F(x_k)-F^*}{2n\Rws{L}{x_0}},1-\tfrac{1}{2n} \right\}(F(x_k)-F^*).\label{eq:tt}
\end{eqnarray}
Let $\xi_k = F(x_k)-F^*$ and consider case (i). If we let $c=2n\max\{\Rws{L}{x_0},F(x_0)-F^*\}$, then from \eqref{eq:tt} we obtain
\[\E[\xi_{k+1} \;|\; \xi_k] \leq (1-\tfrac{\xi_k}{c})\xi_k = \xi_k - \tfrac{\xi_k^2}{c}, \qquad k\geq 0.\]
Moreover,  $\epsilon < \xi_0 < c$. The result then follows by applying Theorem~\ref{l:randomVariableTrick}.
Consider now case (ii). Letting $c= \tfrac{2n\Rws{L}{x_0}}{\epsilon}>1$, notice that if $\xi_k\geq \epsilon$, inequality \eqref{eq:tt} implies that
\[\E[\xi_{k+1} \;|\; \xi_k] \leq \max \left\{1-\tfrac{\epsilon}{2n\Rws{L}{x_0}},1-\tfrac{1}{2n}\right\}\xi_k = (1-\tfrac{1}{c})\xi_k.\]
Again, the result follows from Theorem~\ref{l:randomVariableTrick}.
\end{proof}

\subsection{Strongly Convex Objective}  \label{sec:composite_strong}

Assume that $F$ is strongly convex  with respect to some norm $\|\cdot\|$ with convexity parameter $\mu>0$; that is,
\begin{equation}\label{eq:strong_def}F(x)\geq F(y) + \ve{F'(y)}{x-y} + \tfrac{\mu}{2}\|x-y\|^2, \qquad x,y\in \dom F,\end{equation}
where $F'(y)$ is any subgradient of $F$ at $y$. Note that from the first order optimality conditions for \eqref{eq:P} we obtain $\ve{F'(x^*)}{x-x^*}\geq 0$ for all $x\in \dom F$ which, combining  with \eqref{eq:strong_def} used with $y=x^*$, yields the standard inequality
\begin{equation}\label{eq:stronglyconvexproperty}
 F(x)-F^* \geq \tfrac{\mu}{2} \|x-x^*\|^2, \qquad x\in \dom F.
\end{equation}

The next lemma will be useful in proving linear convergence of the expected value of the objective function to the minimum.

\begin{lemma} \label{lem:Hleq_strongcase}If $F$ is strongly convex with respect to $\|\cdot\|_L$ with convexity parameter $\mu>0$, then
\begin{equation}\label{eq:nonsmooth:gammaxiUsage}
 H(x,T(x)) - F^* \leq
\gamma_\mu
 (F(x)-F^*), \qquad x\in \dom F,
\end{equation}
where
\begin{equation}\label{eq:nonsmooth:gammaxiDef}
 \gamma_\mu =
 \begin{cases}
  1-\tfrac{\mu}{4}, & \text{if } \mu \leq 2, \\
  \tfrac{1}{\mu}, & \text{otherwise. }
 \end{cases}
\end{equation}
\end{lemma}

\begin{proof}
\begin{eqnarray}
\notag H(x,T(x)) &\stackrel{\eqref{eq:T(x)def}}{=}& \min_{t\in \R^{\N}} H(x,t)\\
\notag & = & \min_{y\in\R^{\N}} H(x,y-x)\\
\notag & \leq & \min_{y\in \R^{\N}}   F(y) + \tfrac{1}{2} \|y-x\|_L^2\\
\notag &\leq & \min_{\alpha \in [0,1]}   F(\alpha x^* + (1-\alpha)x) + \tfrac{\alpha^2}{2} \|x-x^*\|_L^2\\
\notag &\leq & \min_{\alpha \in [0,1]} F(x)-\alpha (F(x)-F^*)+ \tfrac{\alpha^2}{2} \|x-x^*\|_L^2\\
\label{eq:asuv1398f0a0sdfa}
&\overset{\eqref{eq:stronglyconvexproperty}}{\leq} & \min_{\alpha \in [0,1]} F(x)+\alpha\left( \tfrac{\alpha}{\mu}-1\right) (F(x)-F^*).
\end{eqnarray}
The optimal $\alpha$ in \eqref{eq:asuv1398f0a0sdfa} is $\alpha^* = \min\left\{1,\tfrac{\mu}{2}\right\}$; the result follows.
\end{proof}

We now show that the expected value of $F(x_k)$ converges to $F^*$ linearly.

\begin{theorem} \label{thm:nonsmooth:12312343} Let $F$ be strongly convex with respect to the norm $\|\cdot\|_L$ with convexity parameter $\mu>0$. If $x_k$ is the random point generated UCDC$(x_0)$, then
  \begin{equation}\label{eq:nonsmmoth:expresultstrongconvex}
  \E [F(x_k)-F^*] \leq \left(1-\tfrac{1- \gamma_\mu}n\right)^k (F(x_0)-F^*),
 \end{equation}
where $\gamma_\mu$ is defined by (\ref{eq:nonsmooth:gammaxiDef}).
\end{theorem}
\begin{proof}Follows from Lemma~\ref{lem:op1} and Lemma~\ref{lem:Hleq_strongcase}.
\end{proof}

The following is an analogue of Theorem~\ref{thm:composite_general_f} in the case of a strongly convex objective. Note that both the accuracy and confidence parameters appear under the logarithm.

\begin{theorem}\label{thm:nonsmooth:stornglyconvex:highprobresult}
Let $F$ be strongly convex with respect to $\|\cdot\|_L$ with convexity parameter $\mu>0$ and choose accuracy level $\epsilon>0$, confidence level $0<\rho<1$, and
\begin{equation}\label{eq:k_uniform_strong}
 k\geq
  \tfrac{n}{ 1- \gamma_\mu } \log \left(\tfrac{F(x_0)-F^*}{\rho\epsilon}\right),
\end{equation}
where $\gamma_\mu$ is given by \eqref{eq:nonsmooth:gammaxiDef}.
If $x_k$ is the random point generated by UCDC$(x_0)$, then
\[\Prob(F(x_k)-F^*\leq \epsilon) \geq 1-\rho.\]
\end{theorem}
\begin{proof}
Using Markov inequality and Theorem~\ref{thm:nonsmooth:12312343}, we obtain
$$
\Prob[F(x_k)-F^*\geq \epsilon]
\leq \tfrac{1}{\epsilon} \E[F(x_k)-F^*]
\overset{\eqref{eq:nonsmmoth:expresultstrongconvex}}{\leq} \tfrac{1}{\epsilon} \left(1-\tfrac{1- \gamma_\mu}{n}\right)^k(F(x_0)-F^*)\stackrel{\eqref{eq:k_uniform_strong}}{\leq} \rho.
$$
\end{proof}

\subsection{A Regularization Technique}\label{sec:regular}

In this part we will investigate an alternative approach to establishing an iteration complexity result in the case of an objective function that is not strongly convex. The strategy is very simple. We first regularize the objective function by adding a small quadratic term to it, thus making it strongly convex, and then argue that when Algorithm~\ref{algorithm:UCD} is applied to the regularized objective, we can recover an approximate solution of the original non-regularized problem.

The result obtained in this way is slightly different to the one covered by Theorem~\ref{thm:composite_general_f} in that $2n\Rws{L}{x_0}$ is replaced by $4n\|x_0-x^*\|_L^2$. In some situations, $\|x_0-x^*\|_L^2$ can be significantly smaller than $\Rws{L}{x_0}$. However, let us remark that the regularizing term depends on quantities that are \emph{not known} in advance.

%
Fix $x_0$ and $\epsilon>0$ and consider a regularized version of the objective function defined by
  \begin{equation} \label{eq:F_mu}F_\mu(x) \eqdef F(x) + \tfrac{\mu}{2} \|x-x_0\|_L^2,
  \qquad   \mu = \tfrac{\epsilon}{\|x_0-x^*\|^2_L}.\end{equation}
Clearly, $F_\mu$ is strongly convex with respect to the norm $\|\cdot\|_L$ with convexity parameter $\mu$. In the rest of this subsection we show that if we apply UCDC$(x_0)$ to $F_\mu$ with target accuracy $\tfrac{\epsilon}{2}$, then with high probability we recover an $\epsilon$-approximate solution of \eqref{eq:P}. We first need to establish that an approximate minimizer of $F_\mu$ must be an approximate minimizer of $F$.

\begin{lemma}\label{lemma:regularizationtrick}

 If $x'$ satisfies $F_\mu(x') \leq \min_{x\in\R^\N} F_\mu(x) +\tfrac{\epsilon}{2}$, then  $F(x') \leq F^* +\epsilon$.
 \end{lemma}

\begin{proof}
Clearly,
\begin{equation}
 \label{eq:asdfasfa}
F(x) \leq F_\mu(x), \qquad x \in \R^\N.
\end{equation}
If we let $x_\mu^{*} \eqdef \arg\min_{x\in\R^\N}F_\mu(x)$, then by assumption,
\begin{equation}\label{eq:nonsmooth:rcdm:makestronglyconvex:2}
  F_\mu(x') -F_\mu(x_\mu^*) \leq \tfrac{\epsilon}{2},
\end{equation}
and
\begin{equation}\label{eq:nonsmooth:rcdm:makestronglyconvex:1}
F_\mu(x_\mu^*)
 = \min_{x\in\R^\N}
 F(x)+\tfrac{\mu}{2} \|x-x_0\|_L^2
 \leq
 F(x^*)+\tfrac{\mu}{2} \|x^*-x_0\|_L^2
 \stackrel{\eqref{eq:F_mu}}{\leq}
 F(x^*) +\tfrac{\epsilon}{2}.
\end{equation}
Putting all these observations together, we get
\[
0
\leq
F(x') -F(x^*)
\overset{\eqref{eq:asdfasfa}}{\leq}
F_\mu(x')-F(x^*)
\overset{\eqref{eq:nonsmooth:rcdm:makestronglyconvex:2}}{\leq}
F_\mu(x_\mu^*)+\tfrac{\epsilon}{2}-F(x^*)
\overset{\eqref{eq:nonsmooth:rcdm:makestronglyconvex:1}}{\leq}
 \epsilon.
\]
\end{proof}

The following theorem is an analogue of Theorem~\ref{thm:composite_general_f}.

\begin{theorem} \label{thm:composite:main}Choose initial point $x_0$, target accuracy \begin{equation}\label{eq:main:eps}0<\epsilon \leq 2\|x_0-x^*\|_L^2,\end{equation} target confidence level $0<\rho<1$, and
\begin{equation}\label{eq:k_uniform} k \geq \tfrac{4n\|x_0-x^*\|_L^2}{\epsilon} \log \left(\tfrac{2(F(x_0)-F^*)}{\rho\epsilon}\right).\end{equation}
If $x_k$ is the random point generated by UCDC$(x_0)$ as applied to $F_\mu$, then
\[\Prob(F(x_k)-F^*\leq \epsilon) \geq 1-\rho.\]
\end{theorem}

\begin{proof}
Let us apply Theorem~\ref{thm:nonsmooth:stornglyconvex:highprobresult} to the problem of minimizing $F_\mu$, composed as $f+\Psi_\mu$, with $\Psi_\mu(x) = \Psi(x)+\tfrac{\mu}{2} \|x-x_0\|_L^2$. Note that
\begin{equation}\label{eq:mainproof1}F_\mu(x_0) - F_\mu(x_\mu^*) \stackrel{\eqref{eq:F_mu}}{=}  F(x_0) - F_\mu(x_\mu^*) \stackrel{\eqref{eq:asdfasfa}}{\leq} F(x_0) - F(x_\mu^*) \leq F(x_0) - F^*,\end{equation}
and
\begin{equation}\label{eq:mainproof2}\tfrac{n}{1-\gamma_\mu} \quad \stackrel{\eqref{eq:nonsmooth:gammaxiDef}, \eqref{eq:F_mu}, \eqref{eq:main:eps}}{=} \quad \tfrac{4n\|x_0-x^*\|_L^2}{\epsilon}.\end{equation}
Comparing \eqref{eq:k_uniform_strong} and \eqref{eq:k_uniform} in view of \eqref{eq:mainproof1} and \eqref{eq:mainproof2}, Theorem~\ref{thm:nonsmooth:stornglyconvex:highprobresult} implies that \[\Prob(F_\mu(x_k) - F_\mu(x_\mu^*) \leq \tfrac{\epsilon}{2}) \geq 1-\rho.\]
It now suffices to apply Lemma \ref{lemma:regularizationtrick}.
\end{proof} 
\section{Coordinate Descent for Smooth Functions} \label{sec:smooth}

In this section we give a much simplified and improved treatment of the smooth case ($\Psi\equiv 0$) as compared to the analysis in Sections 2 and 3 of \cite{Nesterov:2010RCDM}.

As alluded to in the above, we will develop the analysis in the smooth case for arbitrary, possibly non-Euclidean, norms $\|\cdot\|_{(i)}$, $i=1,2,\dots,n$. Let $\|\cdot\|$ be an arbitrary norm in $\R^l$. Then its dual is defined in the usual way:
\[\|s\|^* = \max_{\|t\| = 1} \; \ve{s}{t}.\]

The following (Lemma~\ref{lem:s_sharp}) is a simple result which is used in \cite{Nesterov:2010RCDM} without being fully articulated nor proved as it constitutes a straightforward extension of a fact that is trivial in the Euclidean setting to the case of general norms. Since we think it is perhaps not standard, we believe it deserves to be spelled out explicitly. The lemma has the following use. The main problem which needs to be solved at each iteration of Algorithm~\ref{algorithm:NRCDM} in the smooth case is of the form \eqref{eq:smooth:non-standard-lemma}, with $s=-\tfrac{1}{\Lip_i}\nabla_i f(x_k)$ and $\|\cdot\| = \|\cdot\|_{(i)}$. Since $\|\cdot\|$ is non-Euclidean, we cannot write down the solution of \eqref{eq:smooth:non-standard-lemma} in a closed form a-priori, for all norms. Nevertheless, we can say \emph{something} about the solution, which turns out to be enough for our subsequent analysis.

\begin{lemma}\label{lem:s_sharp}
If by $s^\#$ we denote an optimal solution of the problem
\begin{equation}\label{eq:smooth:non-standard-lemma}\min_t \; \left\{u(s) \eqdef -\ve{s}{t} + \tfrac{1}{2}\|t\|^2 \right\},\end{equation}
then
\begin{equation}
\label{eq:lem:s_sharp1} u(s^\#) =  -\tfrac{1}{2} \left(\|s\|^*\right)^2, \qquad \|s^\#\| = \|s\|^*, \qquad
(\alpha s)^\# = \alpha (s^\#), \; \alpha\in \R.
\end{equation}
\end{lemma}
\begin{proof} For $\alpha=0$ the last statement is trivial. If we fix $\alpha\neq 0$, then clearly
\begin{align*}
u((\alpha s)^\#) = \min_{\|t\|=1} \min_\beta \{-\ve{\alpha s}{\beta t} + \tfrac{1}{2}\|\beta t\|^2\}.
\end{align*}
For fixed $t$ the solution of the inner problem is $\beta = \ve{\alpha s}{t}$, whence
\begin{equation}\label{eq:smooth:lemma_sharp1}u((\alpha s)^\#) = \min_{\|t\|=1} -\tfrac{1}{2}\ve{\alpha s}{t}^2 = -\tfrac{1}{2} \alpha^2 \left(\max_{\|t\|=1} \ve{s}{t}\right)^2 = -\tfrac{1}{2}(\|\alpha s\|^*)^2,\end{equation}
 proving the first claim. Next, note that optimal $t=t^*$ in \eqref{eq:smooth:lemma_sharp1} maximizes $\ve{s}{t}$ over $\|t\|= 1$ and hence $\ve{s}{t^*} = \|s\|^*$, which implies that \[\|(\alpha s)^\#\| = |\beta^*| = |\ve{\alpha s}{t^*}| = |\alpha||\ve{s}{t^*}| = |\alpha|\|s\|^* = \|\alpha s\|^*,\]
giving the second claim. Finally, since $t^*$ depends on $s$ only, we have $(\alpha s)^\# = \beta^* t^* = \ve{\alpha s}{t^*}t^*$ and, in particular, $s^\# = \ve{s}{t^*}t^*$. Therefore, $(\alpha s)^\#  = \alpha (s^\#)$.
\end{proof}

We can use Lemma~\ref{lem:s_sharp} to rewrite the main step of Algorithm~\ref{algorithm:NRCDM} in the smooth case into the more explicit form,
\begin{eqnarray*}T^{(i)}(x) = \arg \min_{t\in \R_i} V_i(x,t) &\stackrel{\eqref{eq:V}}{=}& \arg \min_{t\in \R_i} \ve{\nabla_i f(x)}{t} + \tfrac{\Lip_i}{2}\|t\|_{(i)}^2 \\
&\stackrel{\eqref{eq:smooth:non-standard-lemma}}{=}& \left(-\tfrac{\nabla_i f(x)}{\Lip_i}\right)^\#  \stackrel{\eqref{eq:lem:s_sharp1}}{=} -\tfrac{1}{\Lip_i} (\nabla_i f(x))^\#,\end{eqnarray*}
leading to Algorithm~\ref{algorithm:RCDM-smooth}.

\begin{algorithm}[h!]
\caption{RCDS$(p,x_0)$ (\textbf{R}andomized \textbf{C}oordinate \textbf{D}escent for \textbf{S}mooth Functions)}
\begin{algorithmic} \label{algorithm:RCDM-smooth}
\FOR {$k=0,1,2,\dots$}
 \STATE Choose   $i_{k} = i \in \{1,2,\dots,n\}$ with probability $p_i$
 \STATE $x_{k+1} = x_k -\tfrac{1}{\Lip_{i}} \U_{i}  (\nabla_{i}f(x_k))^\#$
\ENDFOR
\end{algorithmic}
\end{algorithm}

The main utility of Lemma~\ref{lem:s_sharp} for the purpose of the subsequent complexity analysis comes from the fact that it enables us to give an \emph{explicit} bound on the decrease in the objective function during one iteration of the method in the same form as in the Euclidean case:

\begin{eqnarray}
f(x)-f(x+\U_i T^{(i)}(x)) &\stackrel{\eqref{eq:Lipschitz_ineq}}{\geq}& - [ \ve{\nabla_i f(x)}{T^{(i)}(x)} + \tfrac{\Lip_i}{2}\|T^{(i)}(x)\|_{(i)}^2] \notag \\
&=& - \Lip_i u((-\tfrac{\nabla_i f(x)}{\Lip_i})^\#) \label{eq:smooth:decrease}\\
&\stackrel{\eqref{eq:lem:s_sharp1}}{=} &\tfrac{\Lip_i}{2}(\|-\tfrac{\nabla_i f(x)}{\Lip_i}\|_{(i)}^*)^2 = \tfrac{1}{2\Lip_i}(\|\nabla_i f(x)\|_{(i)}^*)^2.\notag
\end{eqnarray}

\subsection{Convex Objective}
We are now ready to state the main result of this section.

\begin{theorem}\label{thm:smooth_main}
Choose initial point $x_0$, target accuracy $0<\epsilon<\min\{f(x_0)-f^*,2\Rws{LP^{-1}}{x_0}\}$, target confidence $0<\rho<1$ and
\begin{equation}\label{eq:k_smooth}
k \geq \tfrac{2\Rws{LP^{-1}}{x_0}}{\epsilon} \left(1 + \log \tfrac{1}{\rho}\right) + 2 - \tfrac{2\Rws{LP^{-1}}{x_0}}{f(x_0)-f^*},\end{equation}
or
\begin{equation}\label{eq:k_smooth2}
k \geq \tfrac{2\Rws{LP^{-1}}{x_0}}{\epsilon} \left(1 + \log \tfrac{1}{\rho}\right) - 2.\end{equation}
If $x_k$ is the random point generated by RCDS$(p,x_0)$ as applied to convex $f$, then
\[\Prob(f(x_k)-f^*\leq \epsilon) \geq 1-\rho.\]
\end{theorem}
\begin{proof}
Let us first estimate the expected decrease of the objective function during one iteration of the method:
\begin{eqnarray*}
f(x_k)-\E[f(x_{k+1}) \;|\; x_k]
&=&\sum_{i=1}^n p_i[f(x_k)-f(x_k+\U_i T^{(i)}(x_k))]\\
&\overset{\eqref{eq:smooth:decrease}}{\geq}& \tfrac{1}{2}\sum_{i=1}^n p_i \tfrac{1}{\Lip_i}   (\nbd{\nabla_i f(x_k)}{i})^2 = \tfrac{1}{2}(\|\nabla f(x_k)\|_W^*)^2,
\end{eqnarray*}
where $W=LP^{-1}$. Since $f(x_k)\leq f(x_0)$ for all $k$ and because $f$ is convex, we get
$f(x_k)-f^* \leq \max_{x^*\in X^*} \la \nabla f(x_k) , x_k - x^* \ra \leq \|\nabla f(x_k)\|_W^*  \Rw{W}{x_0}$, whence
\[
f(x_k)-\E[f(x_{k+1}) \;|\; x_k]
\geq
\tfrac{1}{2}
 \left(\tfrac{f(x_k)-f^*}{\Rw{W}{x_0}}\right)^2.
\]
By rearranging the terms we obtain
$$\E[f(x_{k+1}) - f^* \;|\; x_k]
\leq f(x_k) - f^* -  \tfrac{(f(x_k)-f^*)^2}{2\Rws{W}{x_0}}.
$$
If we now use Theorem~\ref{l:randomVariableTrick} with $\xi_k=f(x_k)-f^*$ and $c =  2\Rws{W}{x_0}$, we obtain the result for $k$ given by \eqref{eq:k_smooth}.
We now claim that $2 - \tfrac{c}{\xi_0} \leq -2$, from which it follows that the result holds for $k$ given by \eqref{eq:k_smooth2}. Indeed, first notice that this inequality is equivalent to
\begin{equation}\label{eq:uio}f(x_0)-f^* \leq \tfrac{1}{2}\Rws{W}{x_0}.\end{equation}
Now, a straightforward extension of Lemma~2 in \cite{Nesterov:2010RCDM} to general weights states that $\nabla f$ is Lipschitz with respect to the norm $\|\cdot\|_V$ with the constant $\tr(LV^{-1})$. This, in turn, implies the inequality \[f(x)-f^* \leq \tfrac{1}{2}\tr(LV^{-1})\|x-x^*\|_V^2,\]
from which \eqref{eq:uio} follows by setting $V=W$ and $x = x_0$.
\end{proof}

\subsection{Strongly Convex Objective}
Assume now that $f$ is strongly convex  with respect to the norm $\|\cdot\|_{LP^{-1}}$ (see definition \eqref{eq:strong_def}) with convexity parameter $\mu>0$. Using \eqref{eq:strong_def} with $x=x^*$ and $y=x_k$, we obtain
\[f^*- f(x_k) \geq \ve{\nabla f(x_k)}{h} + \tfrac{\mu}{2}\|h\|_{LP^{-1}} = \mu\left(\ve{\tfrac{1}{\mu}\nabla f(x_k)}{h} + \tfrac{1}{2}\|h\|_{LP^{-1}}\right),\]
where $h=x^*-x_k$. Applying Lemma~\ref{lem:s_sharp} to estimate the right hand side of the above inequality from below we obtain
\begin{equation}\label{eq:smooth_strong_ineq}f^*- f(x_k) \geq -\tfrac{1}{2\mu}(\|\nabla f(x_k)\|_{LP^{-1}}^*)^2.\end{equation}
Let us now write down an efficiency estimate for the case of a strongly convex objective.

\begin{theorem}\label{thm:smooth_main_strong}
Choose initial point $x_0$, target accuracy $0<\epsilon<f(x_0)-f^*$, target confidence $0<\rho<1$ and
\begin{equation}\label{eq:k_smooth_strong}
k \geq \tfrac{1}{\mu}\log \tfrac{f(x_0)-f^*}{\epsilon\rho}.\end{equation}
If $x_k$ is the random point generated by RCDS$(p,x_0)$ as applied to $f$, then
\[\Prob(f(x_k)-f^*\leq \epsilon) \geq 1-\rho.\]
\end{theorem}
\begin{proof}
Let us first estimate the expected decrease of the objective function during one iteration of the method:
\begin{eqnarray*}
f(x_k)-\E[f(x_{k+1}) \;|\; x_k]
&=&\sum_{i=1}^n p_i[f(x_k)-f(x_k+\U_i T^{(i)}(x_k))]\\
&\stackrel{\eqref{eq:smooth:decrease}}{\geq}& \tfrac{1}{2}\sum_{i=1}^n p_i \tfrac{1}{\Lip_i}   (\nbd{\nabla_i f(x_k)}{i})^2\\
&=& \tfrac{1}{2}(\|\nabla f(x_k)\|_{LP^{-1}}^*)^2\\
&\stackrel{\eqref{eq:smooth_strong_ineq}}{\geq}& \mu(f(x_k)-f^*).
\end{eqnarray*}
After rearranging the terms we obtain
$\E[f(x_{k+1}) - f^* \;|\; x_k]  \leq (1-\mu)\E[f(x_k) - f^*]$. If we now use part (ii) of Theorem~\ref{l:randomVariableTrick}  with $\xi_k=f(x_k)-f^*$ and $c = \tfrac{1}{\mu}$, we obtain the result.
\end{proof}

\section{Comparison of CD Methods with Complexity Guarantees} \label{sec:comparison}

In this section we compare the results obtained in this paper with existing CD methods endowed with iteration complexity bounds.

\subsection{Smooth case ($\Psi = 0$)}

In Table~\ref{T:1} we look at the results for unconstrained smooth minimization of Nesterov~\cite{Nesterov:2010RCDM} and contrast these with our approach. For brevity we only include results for the non-strongly convex case.

\begin{table}[!ht]
\begin{center}
{\footnotesize
\begin{tabular}{|c||c|c|c|c|c|}
  \hline
  Algorithm  & $\Psi$ & $p_i$ & Norms & Complexity & Objective\\
  \hline
  \hline
    &&&&&\\
        \begin{tabular}{c}
     Nesterov \cite{Nesterov:2010RCDM}  \\
     (Theorem 4)  \\
    \end{tabular} & $0$ &  $\tfrac{\Lip_i}{\sum_i \Lip_i}$ & Euclidean &  $(2n+ \tfrac{8\tfrac{\Lip_i}{\sum_i \Lip_i}\Rws{I}{x_0}}{\epsilon})\log \tfrac{4(f(x_0)-f^*)}{\epsilon \rho}$ & $f(x)+\tfrac{\epsilon \|x-x_0\|_{I}^2}{8\Rws{I}{x_0}}$\\
    &&&&&\\
    \begin{tabular}{c}
     Nesterov \cite{Nesterov:2010RCDM}  \\
     (Theorem 3)  \\
    \end{tabular}
     & $0$    & $\tfrac{1}{n}$ & Euclidean &  $\tfrac{8n\Rws{\Lip}{x_0}}{\epsilon}\log \tfrac{4(f(x_0)-f^*)}{\epsilon\rho}$ & $f(x)+\tfrac{\epsilon \|x-x_0\|_\Lip^2}{8\Rws{\Lip}{x_0}}$\\
       &&&&&\\
  \begin{tabular}{c}
  Algorithm~\ref{algorithm:RCDM-smooth} \\
    (Theorem~\ref{thm:smooth_main}) \\
  \end{tabular}
   & $0$    & $>0$ & general &  $\tfrac{2\Rws{LP^{-1}}{x_0}}{\epsilon} (1 + \log \tfrac{1}{\rho}) -2$ & $f(x)$\\ 
  &&&&&\\
  \begin{tabular}{c}
   Algorithm~\ref{algorithm:UCD} \\
    (Theorem~\ref{thm:composite_general_f}) \\
   \end{tabular}
   & separable  & $\tfrac{1}{n}$ & Euclidean & \begin{tabular}{c}
                                                       $\tfrac{2n\max\{\Rws{\Lip}{x_0}, F(x_0)-F^*\}}{\epsilon}(1+\log \tfrac{1}{\rho})$ \\
                                                       $\tfrac{2n\Rws{\Lip}{x_0}}{\epsilon}\log \tfrac{F(x_0)-F^*}{\epsilon \rho}$ \\
                                                     \end{tabular}
    & $F(x)$\\
    &&&&&\\
  \hline
\end{tabular}
}
\end{center}

\caption{Comparison of our results to the results in \cite{Nesterov:2010RCDM} in the  non-strongly convex case. The complexity is for achieving $\Prob(F(x_k)-F^*\leq \epsilon)\geq 1-\rho$.}
\label{T:1}
\end{table}


We will now comment on the contents of Table~\ref{T:1} in detail.

\begin{itemize}
\item \textbf{Uniform probabilities.} 
Note that in the uniform case ($p_i=\tfrac{1}{n}$ for all $i$) we have
\[\Rws{LP^{-1}}{x_0} = n\Rws{L}{x_0},\]
and hence the leading term (ignoring the logarithmic factor) in the complexity estimate of Theorem~\ref{thm:smooth_main} (line 3 of Table~\ref{T:1}) coincides with the leading term in the complexity estimate of Theorem~\ref{thm:composite_general_f} (line 4 of Table~\ref{T:1}; the second result): in both cases it is \[\tfrac{2n\Rws{L}{x_0}}{\epsilon}.\]
Note that the leading term of the complexity estimate given in Theorem~3 of \cite{Nesterov:2010RCDM} (line 2 of Table~\ref{T:1}), which covers the uniform case, is worse by a factor of 4.

\item \textbf{Probabilities proportional to Lipschitz constants.} If we set $p_i = \Lip_i/S$ for all $i$, where $S=\sum_i \Lip_i$, then
\[\Rws{LP^{-1}}{x_0} = S\Rws{I}{x_0}.\]
In this case Theorem 4 in \cite{Nesterov:2010RCDM} (line 1 of Table~\ref{T:1}) gives the complexity bound $2[n+\tfrac{4S\Rws{I}{x_0}}{\epsilon}]$ (ignoring the logarithmic factor), whereas we obtain the bound $\tfrac{2S\Rws{I}{x_0}}{\epsilon}$ (line 3 of Table~\ref{T:1}), an improvement by a factor of 4. Note that there is a further additive decrease by the constant $2n$ (\emph{and} the additional constant $\tfrac{2\Rws{LP^{-1}}{x_0}}{f(x_0)-f^*}-2$ if we look at the sharper bound \eqref{eq:k_smooth}).
\item \textbf{General probabilities.} Note that unlike the results in \cite{Nesterov:2010RCDM}, which cover the choice of two probability vectors only (lines 1 and 2 of Table~\ref{T:1})---uniform and proportional to $\Lip_i$---our result (line 3 of Table~\ref{T:1}) covers the case of arbitrary probability vector $p$. This opens the possibility for fine-tuning the choice of $p$, in certain situations, so as to minimize $\Rws{LP^{-1}}{x_0}$.
\item \textbf{Logarithmic factor.}  Note that in our results we have managed to push $\epsilon$ out of the logarithm.
\item \textbf{Norms.} Our results hold for general norms.
\item \textbf{No need for regularization.} Our results hold for applying the algorithms to $F$ directly; i.e., there is no need to first regularize the function by adding a small quadratic term to it (in a similar fashion as we have done it in Section~\ref{sec:regular}). This is an essential feature as the regularization constants are not known and hence the complexity results obtained that way are not true complexity results.
\end{itemize}

\subsection{Nonsmooth case ($\Psi\neq 0$)}

In Table~\ref{T:2} we summarize the main characteristics of known complexity results for coordinate (or block coordinate) descent methods for minimizing composite functions.

Note that the methods of Saha \& Tewari and Schwarz \& Tewari cover the $\ell_1$ regularized case only, whereas the other methods cover the general block-separable case. However, while the greedy approach of Yun \& Tseng requires per-iteration work which grows with increasing problem dimension, our randomized strategy can be implemented cheaply. This gives an important advantage to randomized methods for problems of large enough size.

The methods of Yun \& Tseng  and Saha \& Tewari  use one Lipschitz constant only, the Lipschitz constant $L(\nabla f)$ of the gradient of $f$. Note that if $n$ is large, this constant is typically much larger than the (block) coordinate constants $\Lip_i$. Schwarz \& Tewari use coordinate Lipschitz constants, but assume that all of them are the same. This is suboptimal as in many applications the constants $\{\Lip_i\}$ will have a large variation and hence if one chooses $\beta = \max_i \Lip_i$ for the common Lipschitz constant, steplengths will necessarily be small (see Figure~\ref{fig:RC2} in Section~\ref{sec:experiments}).

\begin{table}[!h]
\begin{center}
{\footnotesize
\begin{tabular}{|c||c|c|c|c|c|}
\hline
   Algorithm   & \begin{tabular}{c}
                   Lipschitz \\
                   constant(s) \\
                 \end{tabular}
 & $\Psi$    & block   & \begin{tabular}{c}
                   Choice of \\
                   coordinate \\
                 \end{tabular}       & \begin{tabular}{c}
                   Work per \\
                   1 iteration \\
                 \end{tabular}\\
\hline
\hline
&&&&&\\
\begin{tabular}{c}
  Yun \& Tseng \\    \cite{Tseng:2009b:Nonsmooth} \\
\end{tabular}
     & $L(\nabla f)$  & separable                    & Yes   & greedy    &    expensive \\
&&&&&\\
\begin{tabular}{c}
  Saha \& Tewari \\ \cite{Saha10finite} \\
\end{tabular}       & $L(\nabla f)$  & $\|\cdot\|_1$ & No    & cyclic    &  cheap  \\
&&&&&\\
\begin{tabular}{c}
  Shwartz \& Tewari \\ \cite{ShalevTewari09} \\
\end{tabular}
       & $\beta = \max_i \Lip_i$  & $\|\cdot\|_1$ & No         & $\tfrac{1}{n}$    &   cheap  \\
&&&&&\\
\begin{tabular}{c}
  This paper \\ (Algorithm~\ref{algorithm:UCD}) \\
\end{tabular}        & $L_i$ & separable              & Yes        & $\tfrac{1}{n}$      & cheap\\
&&&&&\\
  \hline
\end{tabular}
}
\end{center}
\caption{Comparison of CD approaches for minimizing composite functions (for which iteration complexity results are provided).}
\label{T:2}
\end{table}

Let us now compare the impact of the Lipschitz constants on the complexity estimates.
For simplicity assume $\N=n$ and let $u=x^*-x_0$. The estimates are listed in Table~\ref{T:3}; it is clear from the last column that the the approach with individual constants $\Lip_i$ for each coordinate gives the best complexity.

\begin{table}[!h]
\begin{center}
{\footnotesize
\begin{tabular}{|c||c|c|}
\hline
   Algorithm   &  complexity  & complexity (expanded)  \\
\hline
\hline
&&\\
\begin{tabular}{c}
  Yun \& Tseng \\    \cite{Tseng:2009b:Nonsmooth} \\
\end{tabular}
     & $O(\tfrac{nL(\nabla f)\|x^*-x_0\|^2_2}{\epsilon})$     &  $O(\tfrac{n}{\epsilon}\sum_i L(\nabla f)(u^{(i)})^2)$ \\
&&\\
\begin{tabular}{c}
  Saha \& Tewari \\ \cite{Saha10finite} \\
\end{tabular}
     & $O(\tfrac{nL(\nabla f)\|x^*-x_0\|^2_2}{\epsilon})$      &  $O(\tfrac{n}{\epsilon} \sum_i L(\nabla f)(u^{(i)})^2)$  \\
&&\\
\begin{tabular}{c}
  Shwartz \& Tewari \\ \cite{ShalevTewari09} \\
\end{tabular}
& $O(\tfrac{n \beta \|x^*-x_0\|_2^2}{\epsilon})$    &  $O(\tfrac{n}{\epsilon} \sum_i (\max_i \Lip_i)  (u^{(i)})^2)$  \\
&&\\
\begin{tabular}{c}
  This paper \\ (Algorithm~\ref{algorithm:UCD}) \\
\end{tabular}
   & $O(\tfrac{n\|x^*-x_0\|^2_{L}}{\epsilon})$  &   $O(\tfrac{n}{\epsilon} \sum_i \Lip_i (u^{(i)})^2)$ \\
&&\\
  \hline
\end{tabular}
}
\end{center}
\caption{Comparison of iteration complexities of the methods listed in Table~\ref{T:2}. The complexity in the case of the randomized methods gives iteration counter $k$ for which $\E(F(x_k)\leq\epsilon)$}
\label{T:3}
\end{table}

\section{Numerical Experiments} \label{sec:experiments}

In this section we study the numerical behavior of RCDC on synthetic and real problem instances of two problem classes: Sparse Regression / Lasso (Section~\ref{sec:exp:lasso}) \cite{Tibshirani:1996} and Linear Support Vector Machines (Section~\ref{sec:exp:SVM}). Due to space limitations we will devote a separate report to the study of the (Sparse) Group Lasso problem.

As an important concern in Section~\ref{sec:exp:lasso} is to demonstrate that our methods scale well with size, all experiments were run on a PC with 480GB RAM. All algorithms were written in C.

\subsection{Sparse Regression / Lasso} \label{sec:exp:lasso}

Consider the problem
\begin{equation}\label{eq:sprarseregressionformulation}
  \min_{x\in\R^n} \tfrac{1}{2} \|Ax-b\|_2^2 +\lambda \|x\|_1,
\end{equation}
where $A=[a_1,\dots,a_n]\in\R^{m\times n}$, $b\in \R^m$, and $\lambda\geq 0$. The parameter $\lambda$ is used to induce sparsity in the resulting solution. Note that \eqref{eq:sprarseregressionformulation} is of the form \eqref{eq:P}, with $f(x)=\tfrac{1}{2}\|Ax-b\|_2^2$ and $\Psi(x)=\lambda \|x\|_1$. Moreover, if we let $N=n$ and $U_i = e_i$ for all $i$, then the Lipschitz constants $\Lip_i$ can be computed explicitly: \[\Lip_i = \|a_i\|_2^2.\]
Computation of $t = T^{(i)}(x)$ reduces to the ``soft-thresholding'' operator \cite{Lin:2010:COMS}. In some of the experiments in this section we will allow the probability vector $p$ to change throughout the iterations even though we do not give a theoretical justification for this. With this modification, a direct specialization of RCDC to \eqref{eq:sprarseregressionformulation} takes the form of Algorithm~\ref{algorithm:RCDM-l1Nonsmooth-example}. If uniform probabilities are used throughout, we refer to the method as UCDC.

\begin{algorithm}[h!]
\caption{RCDC for Sparse Regression}
\begin{algorithmic} \label{algorithm:RCDM-l1Nonsmooth-example}
\STATE Choose $x_0 \in \R^n$ and set $g_0=Ax_0-b=-b$
\FOR {$k=0,1,2,\dots$}
 \STATE Choose  $i_{k} = i \in \{1,2,\dots,n\}$ with probability $p_k^{(i)}$
 \STATE $\alpha = a_{i}^T g_k$
 \STATE $\displaystyle t =
 \left\{
\begin{array}{ll}
  -\frac{\alpha+\lambda}{\|a_i\|_2^2}, & \mbox{if}\ \vc{x_k}{{i}}-\frac{\alpha+\lambda}{\|a_i\|_2^2} > 0\\
  -\frac{\alpha-\lambda}{\|a_i\|_2^2}, & \mbox{if}\ \vc{x_k}{{i}}-\frac{\alpha-\lambda}{\|a_i\|_2^2} < 0\\
  -\vc{x_k}{{i}}, &\mbox{otherwise}
 \end{array}
\right.
$
 \STATE $x_{k+1} = x_k + t e_i, \quad g_{k+1} = g_k + t a_{i}$
\ENDFOR
\end{algorithmic}
\end{algorithm}

\nadpis{Instance generator} In order to be able to test Algorithm~\ref{algorithm:RCDM-l1Nonsmooth-example}  under controlled conditions we use a (variant of the) instance generator proposed in Section 6 of \cite{Nesterov:2007composite} (the generator was presented for $\lambda = 1$ but can be easily extended to any $\lambda>0$). In it, one chooses the sparsity level of $A$ and the optimal solution $x^*$; after that $A$, $b$, $x^*$ and $F^*=F(x^*)$ are generated. For details we refer the reader to the aforementioned paper.

In what follows we use the notation $\nnz{A}$ and $\nnz{x}$ to denote the number of nonzero elements of matrix $A$ and of vector $x$, respectively.


\nadpis{Speed versus sparsity} In the first experiment we investigate, on problems of size $m=10^7$ and $n=10^6$, the dependence of the time it takes for UCDC to complete a block of $n$ iterations (the measurements were done by running the method for $10\times n$ iterations and then dividing by 10) on the sparsity levels of $A$ and $x^*$. Looking at Table~\ref{tbl:numebricaltest:sparseregression:speed}, we see that the speed of UCDC depends roughly linearly on the sparsity level of $A$ (and does not depend on $\|x^*\|_0$ at all). Indeed, as $\|A\|_0$ increases from $10^7$ through $10^8$ to $10^9$, the time it takes for the method to complete $n$ iterations increases from about $0.9$s through $4$--$6$s to about $46$ seconds. This is to be expected since the amount of work per iteration of the method in which coordinate $i$ is chosen is proportional to $\|a_i\|_0$ (computation of $\alpha$, $\|a_i\|_2^2$ and $g_{k+1}$).

\begin{table}[htp]
 \centering
 {\footnotesize
 \begin{tabular}[b]{r||r|r|r}
 $\nnz{x^*}$  & $\nnz{A} = 10^7$ & $\nnz{A} = 10^8$ & $\nnz{A} = 10^9$   \\
  \hline \hline
 $16\times 10^2$ &0.89   &5.89    &46.23    \\
 $16\times 10^3$ &0.85   &5.83    &46.07     \\
 $16\times 10^4$ &0.86   &4.28    & 46.93    \\
\end{tabular}
}
 \caption{The time it takes for UCDC to complete a block of $n$ iterations increases linearly with $\|A\|_0$ and does not depend on $\|x^*\|_0$.}
\label{tbl:numebricaltest:sparseregression:speed}
\end{table}


\nadpis{Efficiency on huge-scale problems} Tables~\ref{T:num1} and \ref{T:num2} present typical results of the performance of UCDC, started from $x_0=0$, on synthetic sparse regression instances of big/huge size. The instance in the first table is of size $m=2\times 10^7$ and $n=10^6$, with $A$ having $5\times 10^7$ nonzeros and the support of $x^*$ being of size $160,000$.

\begin{table}[!htp]
 \centering
 {\footnotesize
\begin{tabular}[b]{ccc}
\multicolumn{3}{c}{$A\in\R^{2\cdot10^7\times 10^6}$, $\nnz{A} = 5\cdot10^7$}\\
\begin{tabular}[b]{r||c|r|r}
   ${k}/n$  & $\frac{F(x_k)-F^*}{F(x_0)-F^*}$ & $\nnz{x_k}$ & time [sec] \\
\hline \hline
0.0000  & $10^{0}$  & 0       &   0.0 \\
2.1180 & $10^{-1}$ & 880,056  &   5.6 \\
4.6350  & $10^{-2}$ & 990,166 &  12.3 \\
5.6250  & $10^{-3}$ & 996,121 &  15.1\\
7.9310  & $10^{-4}$ & 998,981 &  20.7\\
10.3920 & $10^{-5}$ & 997,394 &  27.4\\
12.1100 & $10^{-6}$ & 993,569 &  32.3\\
14.4640 & $10^{-7}$ & 977,260 &  38.3\\
18.0720 & $10^{-8}$ & 847,156 &  48.1\\
19.5190 & $10^{-9}$ & 701,449 &  51.7\\
21.4650 & $10^{-10}$ & 413,163 & 56.4 \\
23.9150 & $10^{-11}$ & 210,624 & 63.1 \\
25.1750 & $10^{-12}$ & 179,355 & 66.6 \\
27.3820 & $10^{-13}$ & 163,048 & 72.4 \\
29.9610 & $10^{-14}$ & 160,311 & 79.3 \\
\end{tabular}
&
\phantom{ABC}
&
\begin{tabular}[b]{r||c|r|r}
   ${k}/n$  & $\frac{F(x_k)-F^*}{F(x_0)-F^*}$ & $\nnz{x_k}$ & time [sec] \\
\hline \hline
30.9440 & $10^{-15}$ & 160,139 & 82.0 \\
32.7480 & $10^{-16}$ & 160,021 & 86.6 \\
34.1740 & $10^{-17}$ & 160,003 & 90.1 \\
35.2550 & $10^{-18}$ & 160,000 & 93.0 \\
36.5480 & $10^{-19}$ & 160,000 & 96.6 \\
38.5210 & $10^{-20}$ & 160,000 &101.4  \\
39.9860 & $10^{-21}$ & 160,000 &105.3  \\
40.9770 & $10^{-22}$ & 160,000 &108.1  \\
43.1390 & $10^{-23}$ & 160,000 &113.7  \\
47.2780 & $10^{-24}$ & 160,000 &124.8  \\
47.2790 & $10^{-25}$ & 160,000 &124.8  \\
47.9580 & $10^{-26}$ & 160,000 &126.4  \\
49.5840 & $10^{-27}$ & 160,000 &130.3  \\
52.3130 & $10^{-28}$ & 160,000 &136.8  \\
53.4310 & $10^{-29}$ & 160,000 &139.4  \\
\end{tabular}
\end{tabular}
}
\caption{Performance of UCDC on a sparse regression instance with a million variables.}
\label{T:num1}
\end{table}

In both tables the first column corresponds to the ``full-pass'' iteration counter $k/n$. That is, after $k=n$ coordinate iterations the value of this counter is 1, reflecting a single ``pass'' through the coordinates. The remaining columns correspond to, respectively, the size of the current residual $F(x_k)-F^*$ relative to the initial residual $F(x_0)-F^*$, size $\nnz{x_k}$ of the support of the current iterate $x_k$, and time (in seconds). A row is added whenever the residual initial residual is decreased by an additional factor of 10.

Let us first look at the smaller of the two problems (Table~\ref{T:num1}). After $35\times n$ coordinate iterations, UCDC decreases the initial residual by a factor of $10^{18}$, and this takes about a minute and a half. Note that the number of nonzeros of $x_k$ has stabilized  at this point at $160,000$, the support size of the optima solution. The method has managed to identify the support. After 139.4 seconds the residual is decreased by a factor of $10^{29}$. This surprising convergence speed can in part be explained by the fact that for random instances with $m>n$, $f$ will typically be strongly convex, in which case UCDC converges linearly (Theorem~\ref{thm:nonsmooth:stornglyconvex:highprobresult}).

UCDC has a very similar behavior on the larger problem as well (Table~\ref{T:num2}). Note that $A$ has 20 billion nonzeros. In $1\times n$ iterations the initial residual is decreased by a factor of $10$, and this takes less than an hour and a half. After less than a day,  the residual is decreased by a factor of 1000. Note that it is very unusual for convex optimization methods equipped with iteration complexity guarantees to be able to solve problems of these sizes.

\begin{table}
 \centering
 {\footnotesize
\begin{tabular}[b]{c}
$A\in\R^{10^{10}\times 10^9}$, $\nnz{A} = 2\times10^{10}$\\
 \begin{tabular}[b]{r||c|r|r}
   ${k}/n$  & $\frac{F(x_k)-F^*}{F(x_0)-F^*}$ & $\nnz{x_k}$ & time [hours] \\
\hline \hline
  0     & $10^{0}$ & 0        &   0.00 \\
  1 & $10^{-1}$ & 14,923,993&  1.43  \\ 
  3 & $10^{-2}$ & 22,688,665&  4.25  \\ 
 16 & $10^{-3}$ & 24,090,068&  22.65 \\ 
\end{tabular}
\end{tabular}
}
\caption{Performance of UCDC on a sparse regression instance with a billion variables and 20 billion nonzeros in matrix $A$.}
\label{T:num2}
\end{table}


\nadpis{Performance on fat matrices ($m<n$)} When $m<n$, then $f$ is not strongly convex and UCDC has the complexity $O(\tfrac{n}{\epsilon}\log \tfrac{1}{\rho})$ (Theorem~\ref{thm:composite_general_f}). In Table~\ref{T:num1N<M} we illustrate the behavior of the method on such an instance; we have chosen $m =10^4$, $n= 10^5$, $\nnz{A} = 10^7$ and $\nnz{x^*}=1,600$. Note that after the first $5,010\times n$ iterations UCDC decreases the residual by a factor of 10+ only; this takes less than $19$ minutes. However, the decrease from $10^2$ to $10^{-3}$ is done in $15\times n$ iterations and takes 3 seconds only, suggesting very fast local convergence.

\begin{table}[!htp]
 \centering
 {\footnotesize
 \begin{tabular}[b]{r||c|r|r}
   ${k}/n$  & $F(x_k)-F^*$ & $\nnz{x_k}$ & time [s] \\
\hline \hline
$1$   &$>10^7$ &63,106&0.21 \\
$5,010$&$<10^6$&33,182& 1,092.59\\
$18,286$&$<10^5$&17,073& 3,811.67 \\
$21,092$&$<10^4$&15,077& 4,341.52\\
$21,416$&$<10^3$&11,469& 4,402.77\\
$21,454$&$<10^2$&5,316& 4,410.09\\
$21,459$&$<10^1$&1,856& 4,411.04\\
$21,462$&$<10^0$&1,609& 4,411.63\\
$21,465$&$<10^{-1}$&1,600& 4,412.21\\
$21,468$&$<10^{-2}$&1,600& 4,412.79\\
$21,471$&$<10^{-3}$&1,600& 4,413.38\\
\end{tabular}
}
 \caption{UCDC needs many more iterations when $m<n$, but local convergence is still fast.}\label{T:num1N<M}
\end{table}


\nadpis{Comparing different probability vectors}

Nesterov \cite{Nesterov:2010RCDM} considers only probabilities proportional to a power of the Lipschitz constants: \begin{equation}\label{eq:exp:alpha}p_i  = \tfrac{\Lip_i^\alpha}{\sum_{i=1}^n \Lip_i^\alpha}, \qquad 0\leq \alpha\leq 1.\end{equation}
In Figure~\ref{fig:RC3} we compare the behavior of RCDC, with the probability vector chosen according to the power law \eqref{eq:exp:alpha}, for three different values of $\alpha$ (0, 0.5 and 1). All variants of RCDC were compared on a single instance with $m=1,000$, $n=2,000$ and $\nnz{x^*}=300$ (different instances produced by the generator yield similar results) and with $\lambda \in \{0,1\}$. The plot on the left corresponds to $\lambda=0$, the plot on the right to $\lambda=1$.

\begin{figure}[!htp]
 \centering
 \includegraphics[width=14cm, height = 5.5cm]{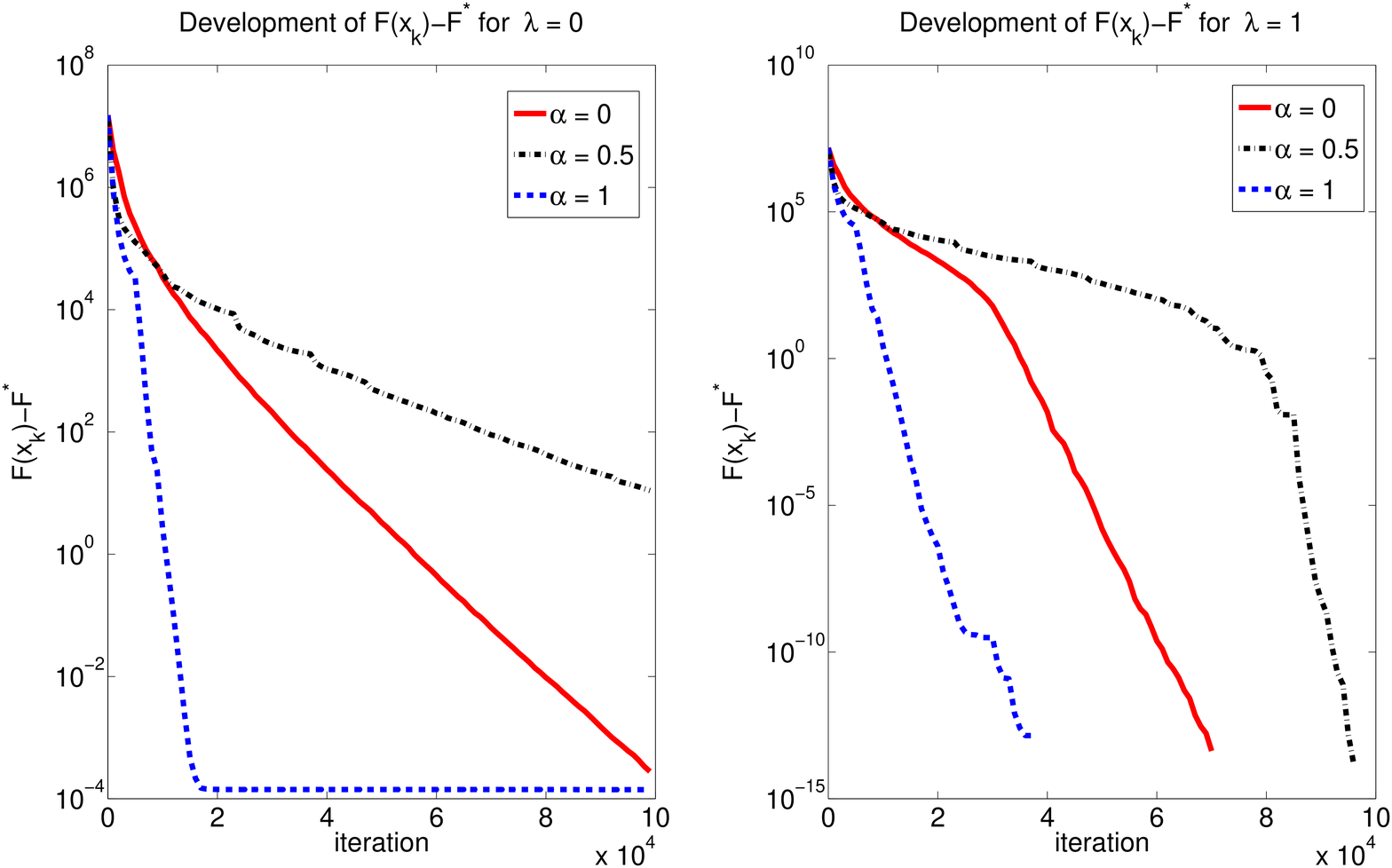}
 \caption{Development of $F(x_k)-F^*$ for sparse regression problem with $\lambda=0$ (left)
 and $\lambda=1$ (right).}
 \label{fig:RC3}
\end{figure}

Note that in both cases the choice $\alpha=1$ is the best. In other words, coordinates with large $\Lip_i$ have a tendency to decrease the objective   function the most. However, looking at the $\lambda=0$ case, we see that the method with $\alpha = 1$ stalls after about 20,000 iterations. The reason for this is that  now the coordinates with small $\Lip_i$ should be chosen to further decrease the objective  value. However, they are chosen with very small probability and hence the slowdown. A solution to this could be to start the method with $\alpha = 1$ and then switch to $\alpha=0$ later on. On the problem with $\lambda =1$ this effect is less pronounced. This is to be expected as now the objective function is a combination of $f$ and $\Psi$, with $\Psi$ exerting its influence and mitigating the effect of the Lipschitz constants.

\nadpis{Coordinate Descent vs. a Full-Gradient method}  In Figure~\ref{fig:RC3} we compare the performance of RCDC with the full gradient (FG) algorithm \cite{Nesterov:2007composite} (with the Lipschitz constant $L_{FG} = \lambda_{\text{max}} (A^TA) > \max_i{\Lip_i}$) for four different distributions of the Lipschitz constants $\Lip_i$. Since the work performed during one iteration of FG is comparable with the work performed by UCDC during $n$ coordinate iterations, for FG we multiply the iteration count by $n$. In all four tests we solve instances with $A \in \R^{2,000\times 1,000}$.

In the 1-1 plot the Lipschitz constants $\Lip_i$ were generated uniformly at random in the interval $(0,1)$. We see that the RCDC variants with $\alpha=0$ and $\alpha=0.2$ exhibit virtually the same behavior, whereas $\alpha=1$ and FG struggle finding a solution with error tolerance below $10^{-5}$ and $10^{-2}$, respectively. The $\alpha=1$ method does start off a bit faster, but then stalls due to the fact that the coordinates with small Lipschitz constants are chosen with extremely small probabilities. For a more accurate solution one needs to be updating these coordinates as well.

In order to zoom in on this phenomenon, in the 1-2 plot we construct an instance with an extreme distribution of Lipschitz constants:
98\% of the constants have the value $10^{-6}$, whereas the remaining 2\% have the value $10^3$. Note that while the FG and $\alpha=1$ methods are able to quickly decrease the objective function within $10^{-4}$ of the optimum, they get stuck afterwards since they effectively never update the coordinates with $\Lip_i = 10^{-6}$. On the other hand, the $\alpha = 0$ method starts off slowly, but does not stop and manages to solve the problem eventually, in about $2\times 10^5$ iterations.

\begin{figure}[!htp]
 \centering
 \begin{tabular}{c|c}
\includegraphics[width=7cm]{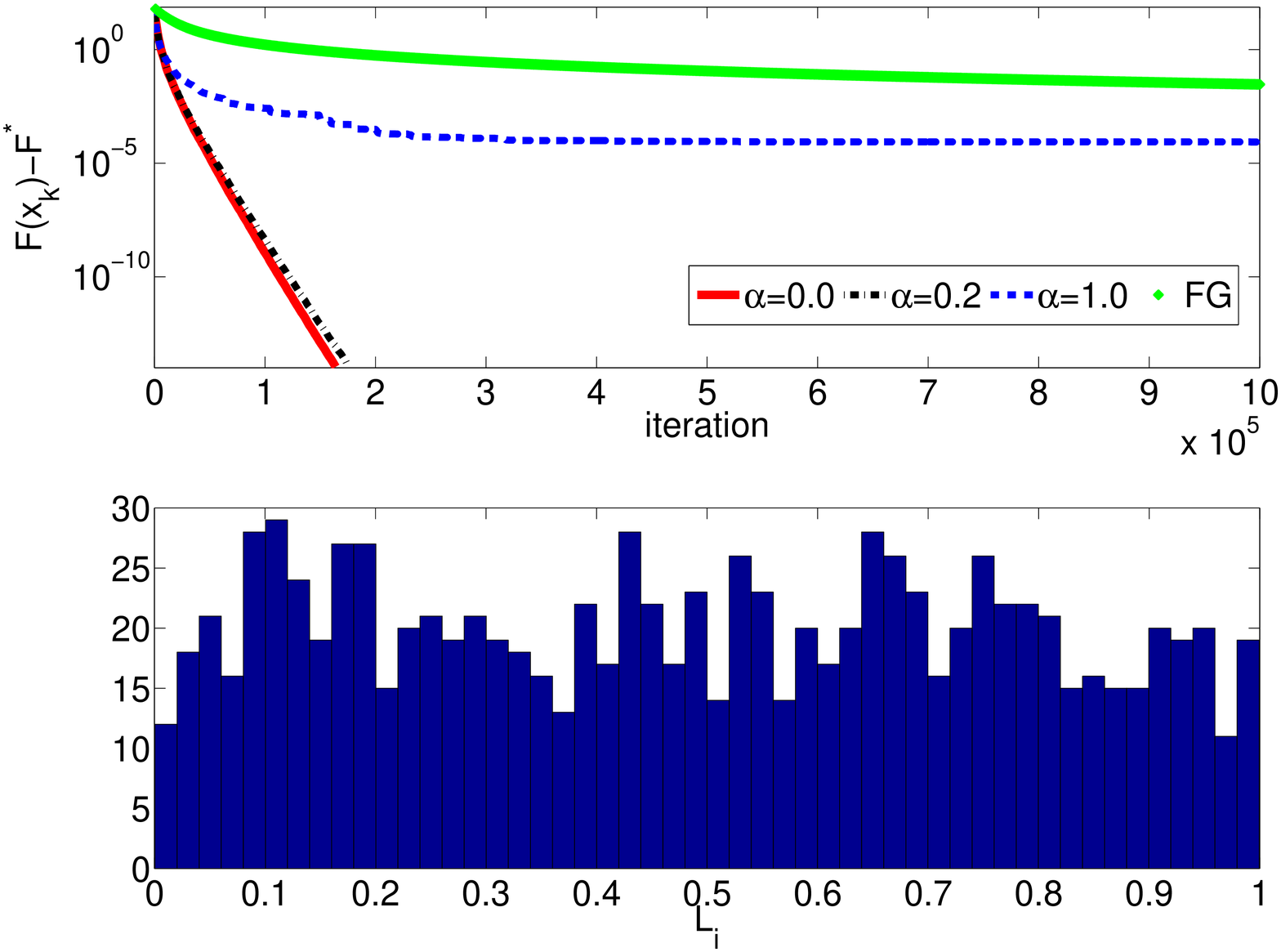}
     &
\includegraphics[width=7cm]{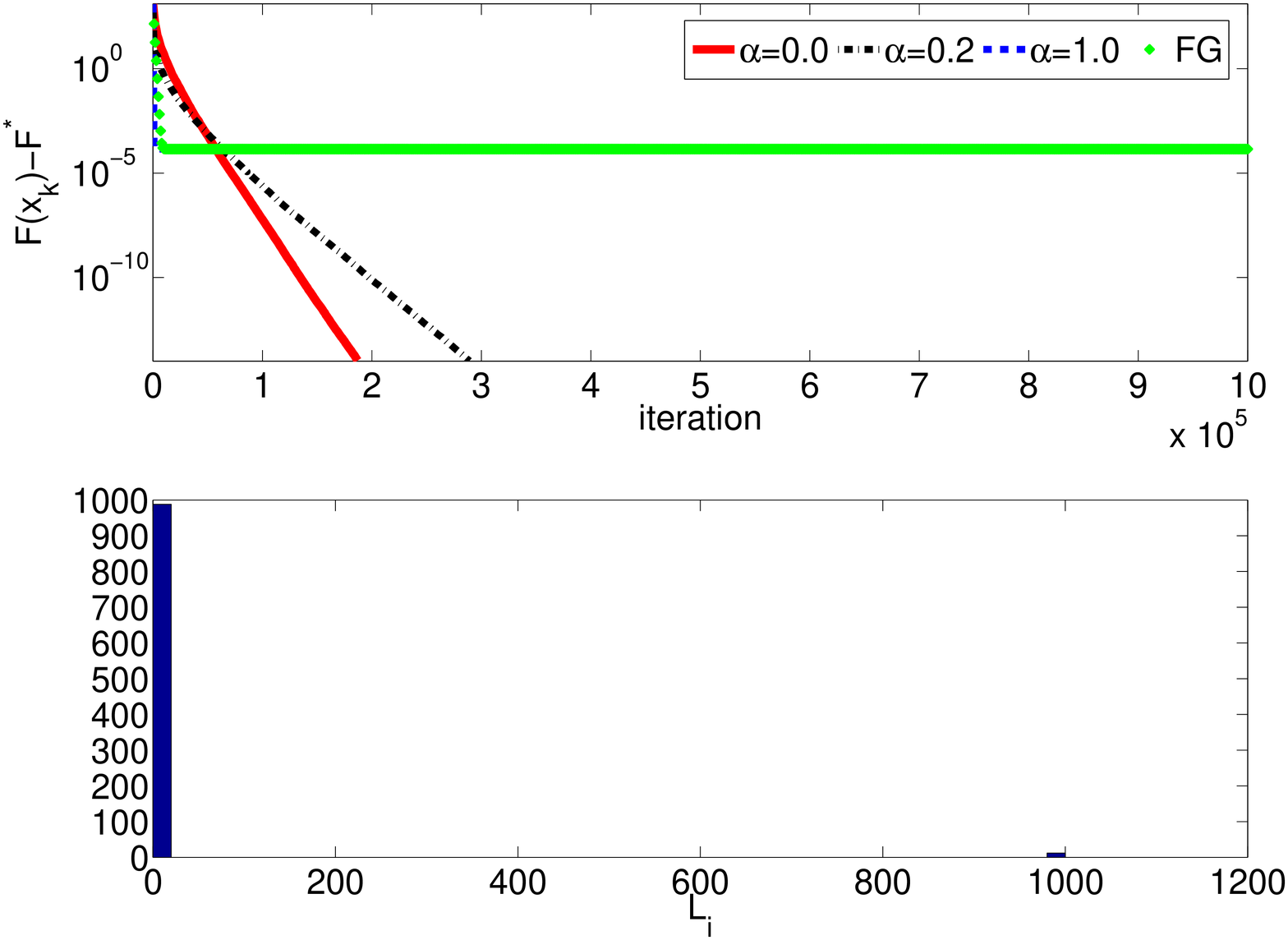}
      \\
\hline
\\
\includegraphics[width=7cm]{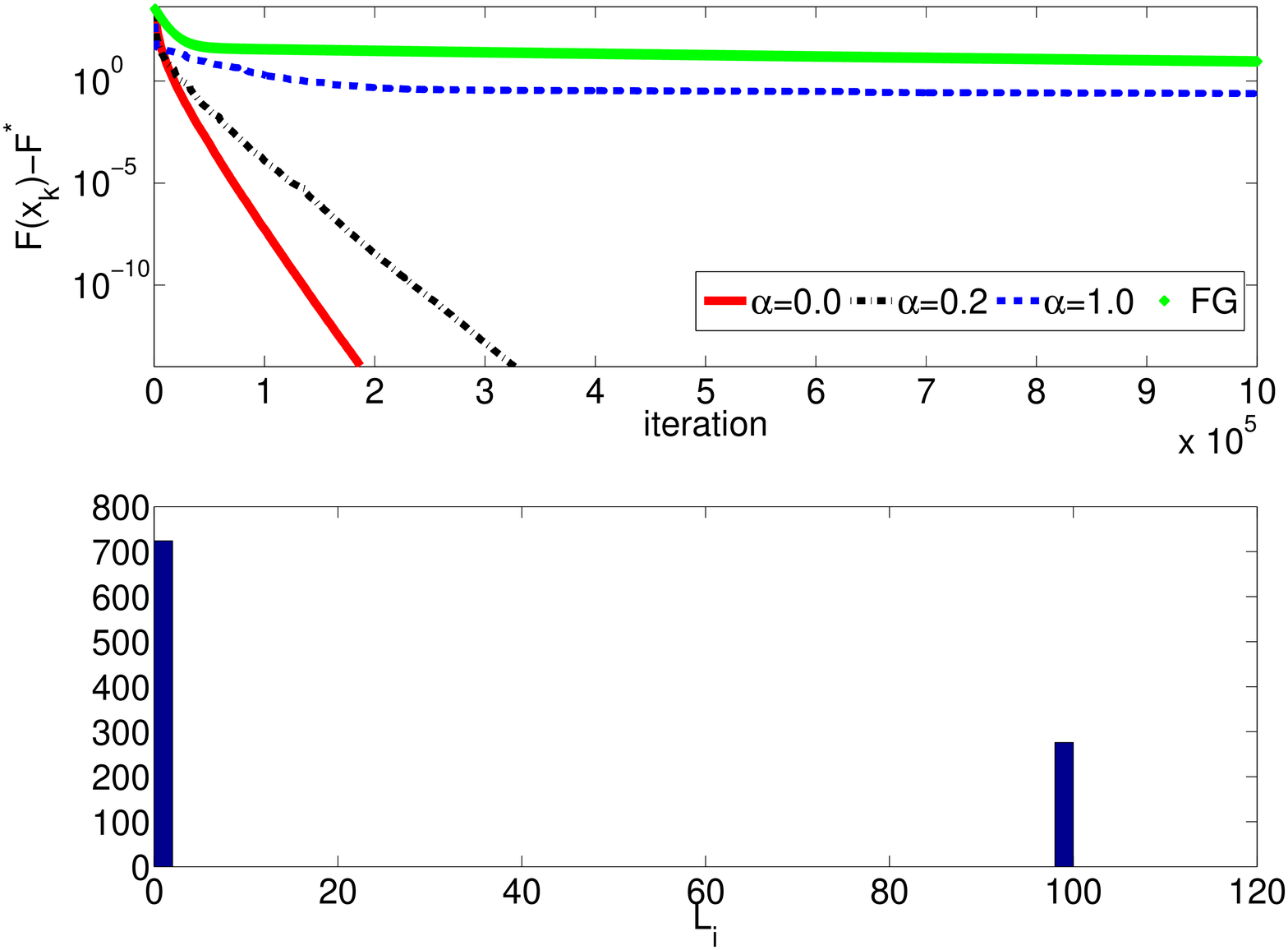}
   &
\includegraphics[width=7cm]{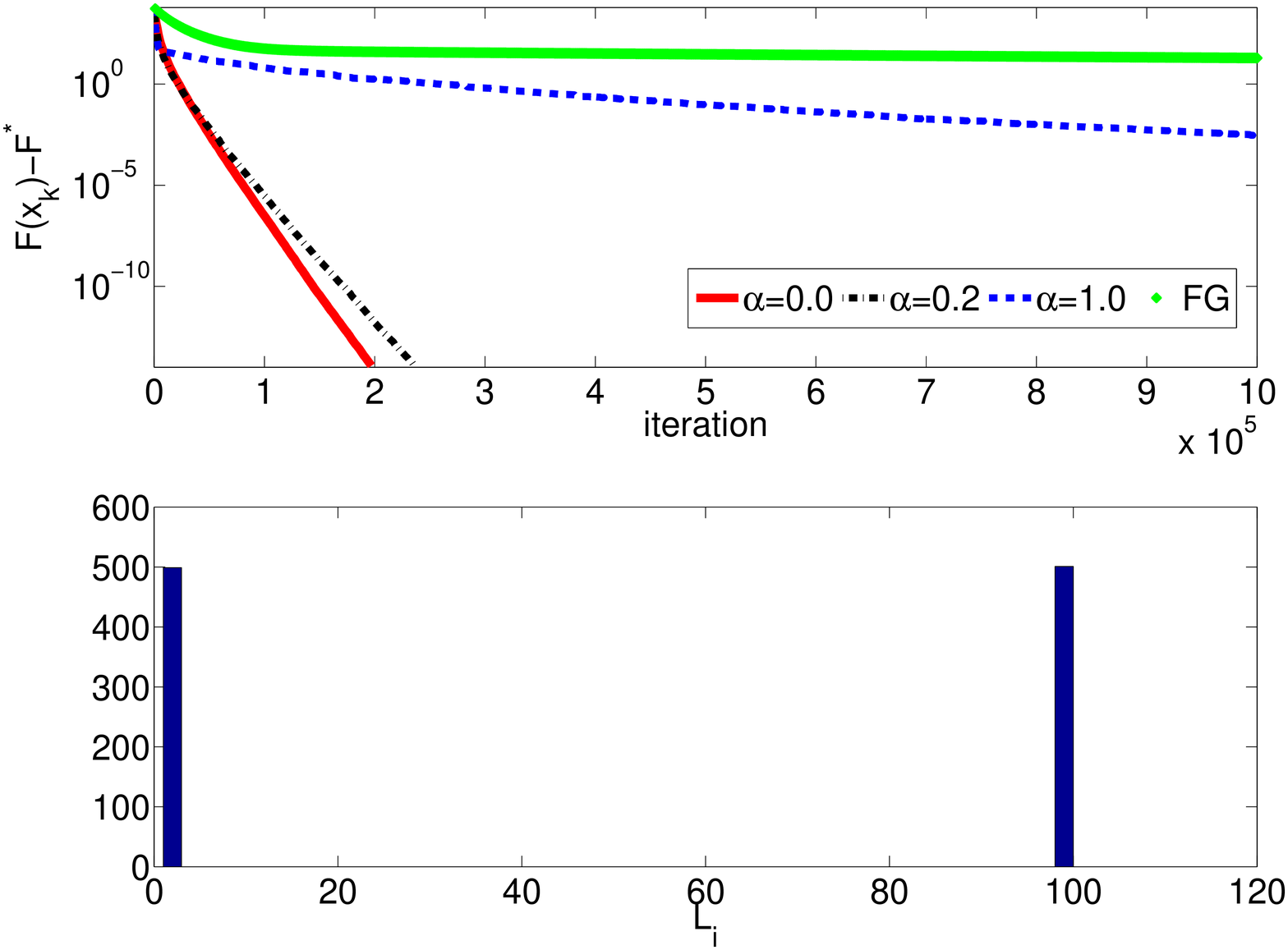}
\end{tabular}
\caption{Comparison UCDC with different choices of $\alpha$ with a full-gradient method (which essentially is UCDC with one component: $n=1$) for four different distributions of the Lipschitz constants $\Lip_i$.}
\label{fig:RC2}
\end{figure}

In the plot in the  2-1 position (resp. 2-2 position) we choose 70\% (resp. 50\%) of the Lipschitz constants $\Lip_i$ to be 1, and the remaining 30\% (resp. 50\%) equal to 100. Again, the $\alpha=0$ and $\alpha = 0.2$ methods give the best long-term performance.

In summary, if fast convergence to a solution with a moderate accuracy us needed, then $\alpha=1$ is the best choice (and is always better than FG). If one desires a solution of higher accuracy, it is recommended to switch to $\alpha = 0$. In fact, it turns out that we can do much better than this using a ``shrinking'' heuristic.


\nadpis{Speedup by shrinking} It is well-known that increasing values of $\lambda$ encourage increased sparsity in the solution of \eqref{eq:sprarseregressionformulation}. In the experimental setup of this section we observe that from certain iteration onwards, the sparsity pattern of the iterates of RCDC is a very good predictor of the sparsity pattern of the optimal solution $x^*$ the iterates converge to. More specifically, we often observe in numerical experiments that for large enough $k$ the following holds:
\begin{equation}\label{eq:shrinking_formula}(x_k^{(i)} = 0) \quad \Rightarrow \quad (\forall l\geq k \quad x_l^{(i)} = (x^*)^{(i)} = 0).\end{equation}
In words, for large enough $k$, zeros in $x_k$ typically stay zeros in all subsequent iterates\footnote{There are various theoretical results on the \emph{identification of active manifolds} explaining numerical observations of this type; see \cite{LewisWright:Composite} and the references therein. See also \cite{Lin:2010:COMS}.} and correspond to zeros in $x^*$. Note that RCDC is \emph{not} able to take advantage of this. Indeed, RCDC, as presented in the theoretical sections of this paper, uses the fixed probability vector $p$ to randomly pick a single coordinate $i$ to be updated in each iteration. Hence, eventually, $\sum_{i:x_k^{(i)}=0} p_i$ proportion of time will be spent on vacuous updates.

Looking at the data in Table~\ref{T:num1} one can see that after approximately $35\times n$ iterations, $x_k$ has the same number of non-zeros as $x^*$ (160,000). What is not visible in the table is that, in fact, the relation \eqref{eq:shrinking_formula} holds for this instance much sooner.
In Figure~\ref{fig:shrinking} we illustrate this phenomenon in more detail on an instance with $m=500$, $n=1,000$ and $\nnz{x^*}=100$.

\begin{figure}[!htp]
 \centering
 \includegraphics[width=7cm]{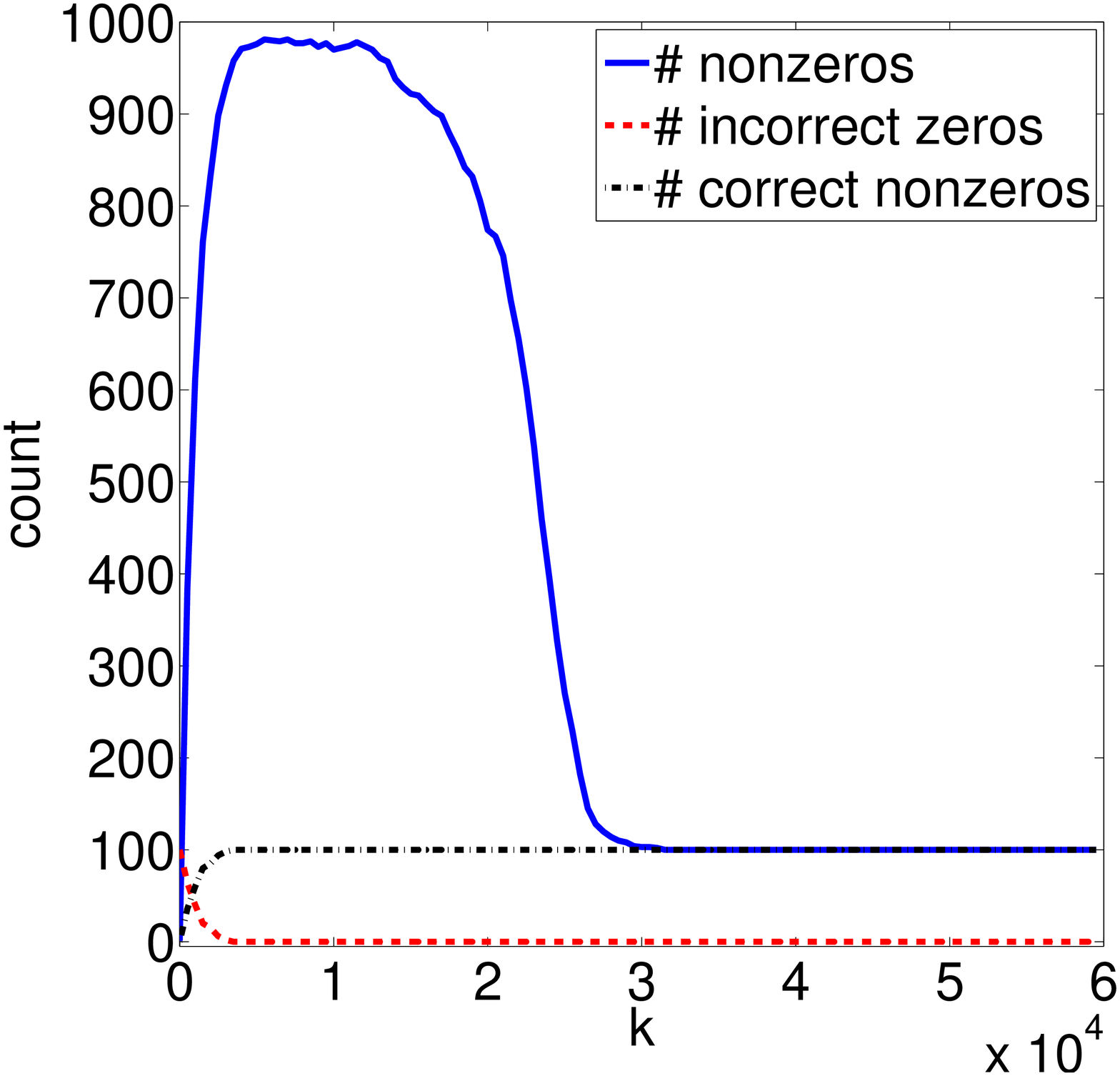}
 \caption{Development of non-zero elements in $x_k$.}
 \label{fig:shrinking}
\end{figure}

First, note that the \emph{number of nonzeros} (solid blue line) in the current iterate, $\# \{i: x_k^{(i)} \neq 0\}$, is first growing from zero (since we start with $x_0=0$) to just below $n$ in about $0.6\times 10^4$ iterations. This value than starts to decrease starting from about $k\approx 15n$ and reaches the optimal number of nonzeros at iteration $k\approx 30n$ and stays there afterwards. Note that the number of \emph{correct nonzeros}, \[cn_k = \# \{i: x_k^{(i)} \neq 0 \; \& \; (x^*)^{(i)}\neq 0 \},\] is increasing (for this particular instance) and reaches the optimal level $\|x^*\|_0$ very quickly (at around $k\approx 3n$). An alternative, and perhaps a more natural, way to look at the same thing is via the  \emph{number of incorrect zeros},
\[iz_k = \# \{i: x_k^{(i)} = 0\; \& \;(x^*)^{(i)}\neq 0 \}.\]
Indeed, we have $cn_k + iz_k = \|x^*\|_0$. Note that for our problem $iz_k \approx 0$ for $k\geq k_0\approx 3n$.

The above discussion suggests that an \emph{iterate-dependent} policy for updating of the probability vectors $p_k$ in Algorithm~\ref{algorithm:RCDM-l1Nonsmooth-example} might help to accelerate the method. Let us now introduce a simple \emph{$q$-shrinking} strategy for adaptively changing the probabilities as follows: at iteration $k\geq k_0$, where $k_0$ is large enough, set
$$
  \vc{p_k}{i} = \vc{\hat{p}_k}{i}(q) \eqdef \left\{
    \begin{array}{ll}
      \frac{1-q}n, & \mbox{if} \ \vc{x_k}{i} = 0, \\
      \frac{1-q}n + \frac q{\nnz{x_k}}, & \mbox{otherwise}.
    \end{array}
  \right.
$$
This is equivalent to choosing $i_k$ uniformly from the set $\{1,2,\dots,n\}$ with probability $1-q$ and uniformly from the support set of $x_k$ with probability $q$. Clearly, different variants of this can be implemented, such as fixing a new probability vector for $k\geq k_0$ (as opposed to changing it for every $k$) ; and some may be more effective and/or efficient than others in a particular context. In Figure~\ref{fig:qshrinking} we illustrate the effectiveness of $q$-shrinking on an instance of size $m=500$, $n=1,000$ with $\nnz{x^*}=50$. We apply to this problem a modified version of RCDC started from the origin ($x_0=0$) in which uniform probabilities are used in iterations $0,\dots,k_0-1$, and $q$-shrinking is introduced as of iteration $k_0$:
\[p_k^{(i)} = \begin{cases}\tfrac{1}{n},       &\text{ for } \quad k = 0,1,\dots,k_0-1,\\
                           \hat{p}_k^{(i)}(q), & \text{ for } \quad k\geq k_0.
\end{cases}\]
We have used $k_0 = 5\times n$.

\begin{figure}[!htp]
 \centering
  \includegraphics[width=7cm]{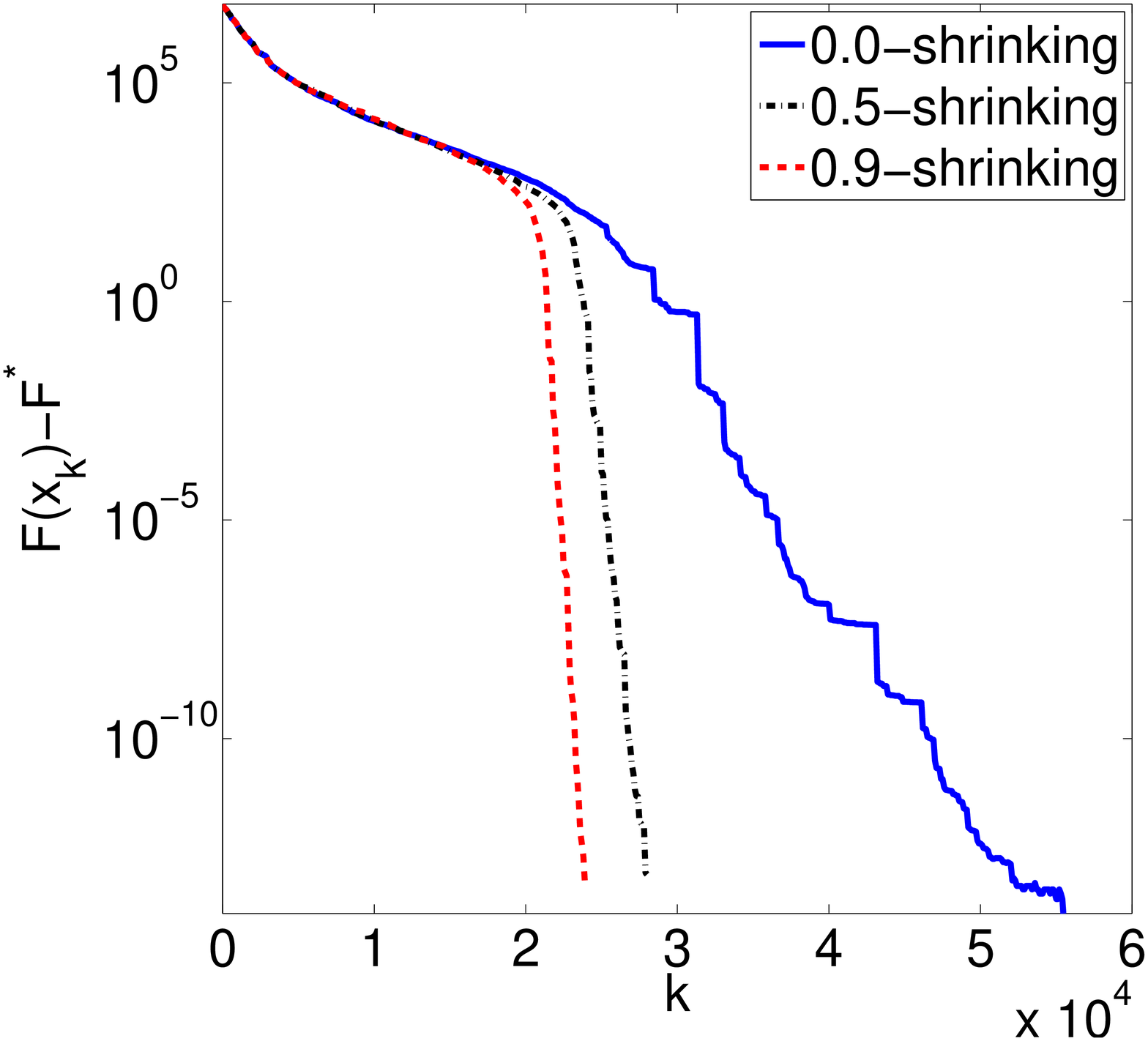}
  \includegraphics[width=7cm]{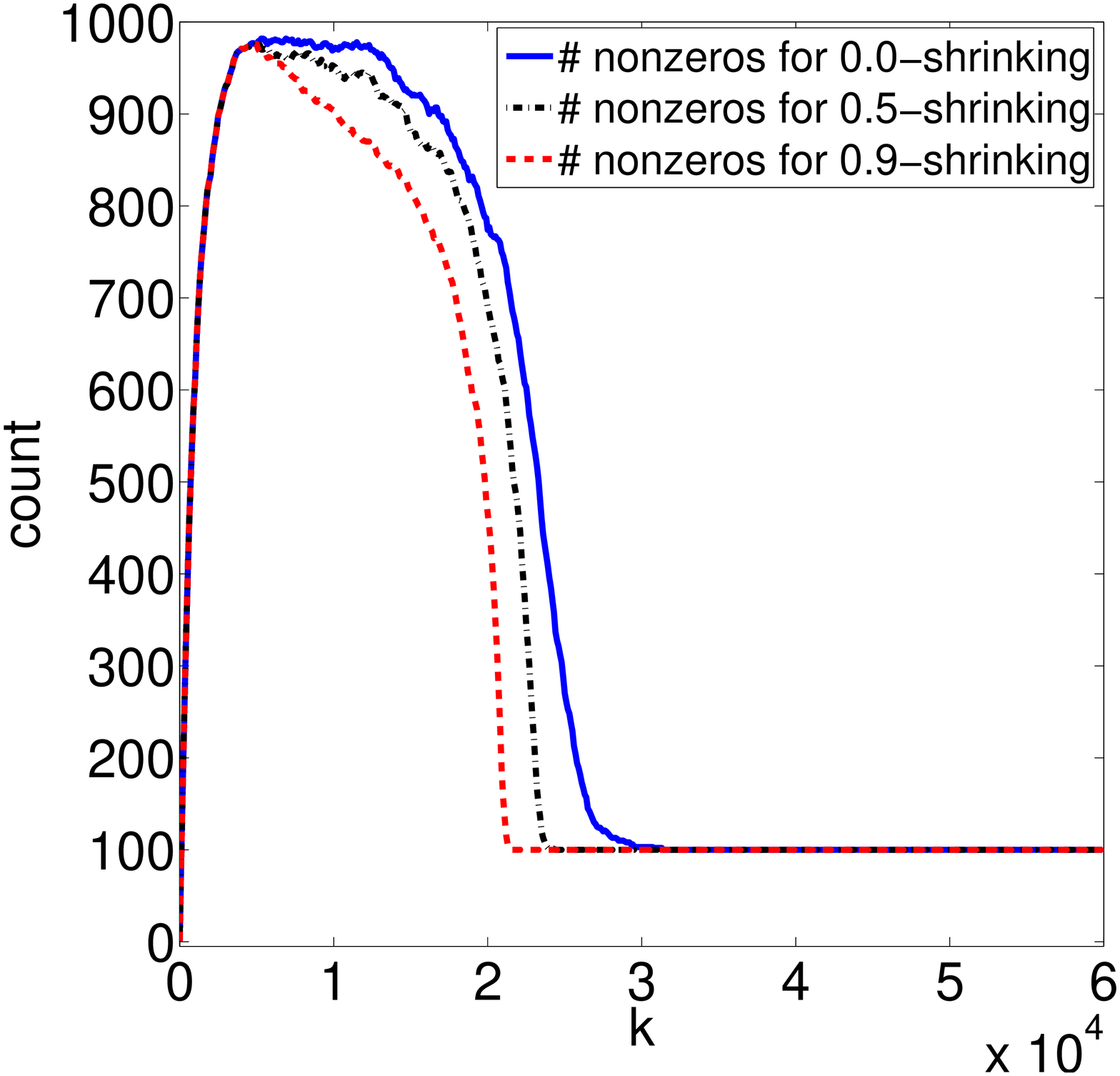}
 \caption{Comparison of different shrinking strategies.}
 \label{fig:qshrinking}
\end{figure}

Notice that as the number of nonzero elements of $x_k$ decreases, the time savings from $q$-shrinking grow. Indeed, $0.9$-shrinking introduces a saving of nearly 70\% when compared to $0$-shrinking to obtain $x_k$ satisfying $F(x_k)-F^*\leq 10^{-14}$. We have repeated this experiment with two modifications: a) a random point was used as the initial iterate (scaled so that $\nnz{x_0} = n$) and b) $k_0=0$. The corresponding plots are very similar to Figure~\ref{fig:qshrinking} with the exception that the lines in the second plot start from $\|x_0\|_0 = n$.


\nadpis{Choice of the initial point} Let us now investigate the question of the choice of the initial iterate $x_0$ for RCDC. Two choices seem very natural: a) $x_0=0$ (the minimizer of $\Psi(x)=\lambda \|x\|_1$) and b) $x_0 = x_{LS}$ (the minimizer of $f(x)=\tfrac{1}{2}\|Ax-b\|_2^2$). Note that the computation of $x_{LS}$ may be as complex as the solution of the original problem. However, if available, $x_{LS}$ constitutes a reasonable alternative to $0$: intuitively, the former will  be preferable to the latter whenever $\Psi$ is dominated by $f$, i.e., when $\lambda$ is small.

In Figure~\ref{fig:warmstartws} we compare the performance of UCDC run on a single instance when started from these two starting points (the solid line corresponds to $x_0=0$ whereas the dashed line corresponds to $x_0 = x_{LS}$). The same instance is used here as in the $q$-shrinking experiments and $\lambda=1$. Starting from $x_{LS}$ gives a $4\times$ speedup for pushing residual $F(x_k)-F^*$ below $10^{-5}$.

\begin{figure}[!htp]
 \centering
 \includegraphics[width=7cm]{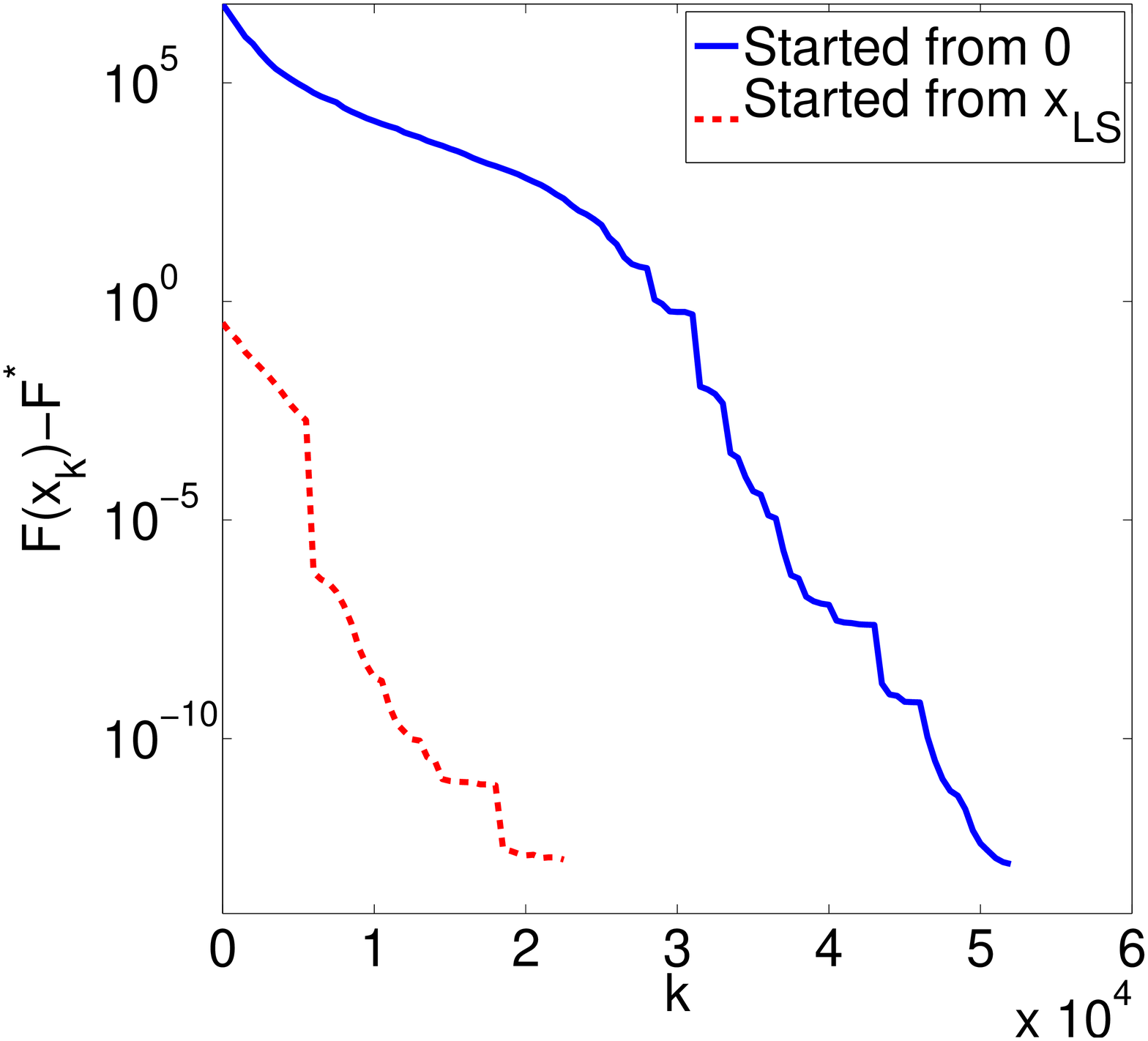}
\caption{Starting from the least squares solution, if available, and if $\lambda$ is small enough, can be better than starting from the origin.}
 \label{fig:warmstartws}
\end{figure}

In Figure~\ref{fig:warmstartwscomaprerandompoint} we investigate the effect of starting UCDC from a point on the line segment between $x_{LS}$ (dashed red line) and $0$ (solid blue line). We generate 50 such points $x_0$, uniformly at random (thin green lines). The plot on the left corresponds to the choice $\lambda = 0.01$, the plot on the right to $\lambda = 1$. Note that $x^*=x_{LS}$ for $\lambda = 0$ and $x^*=0$ when $\lambda \to \infty$.
\begin{figure}[!htp]
 \centering
 \includegraphics[width=7cm]{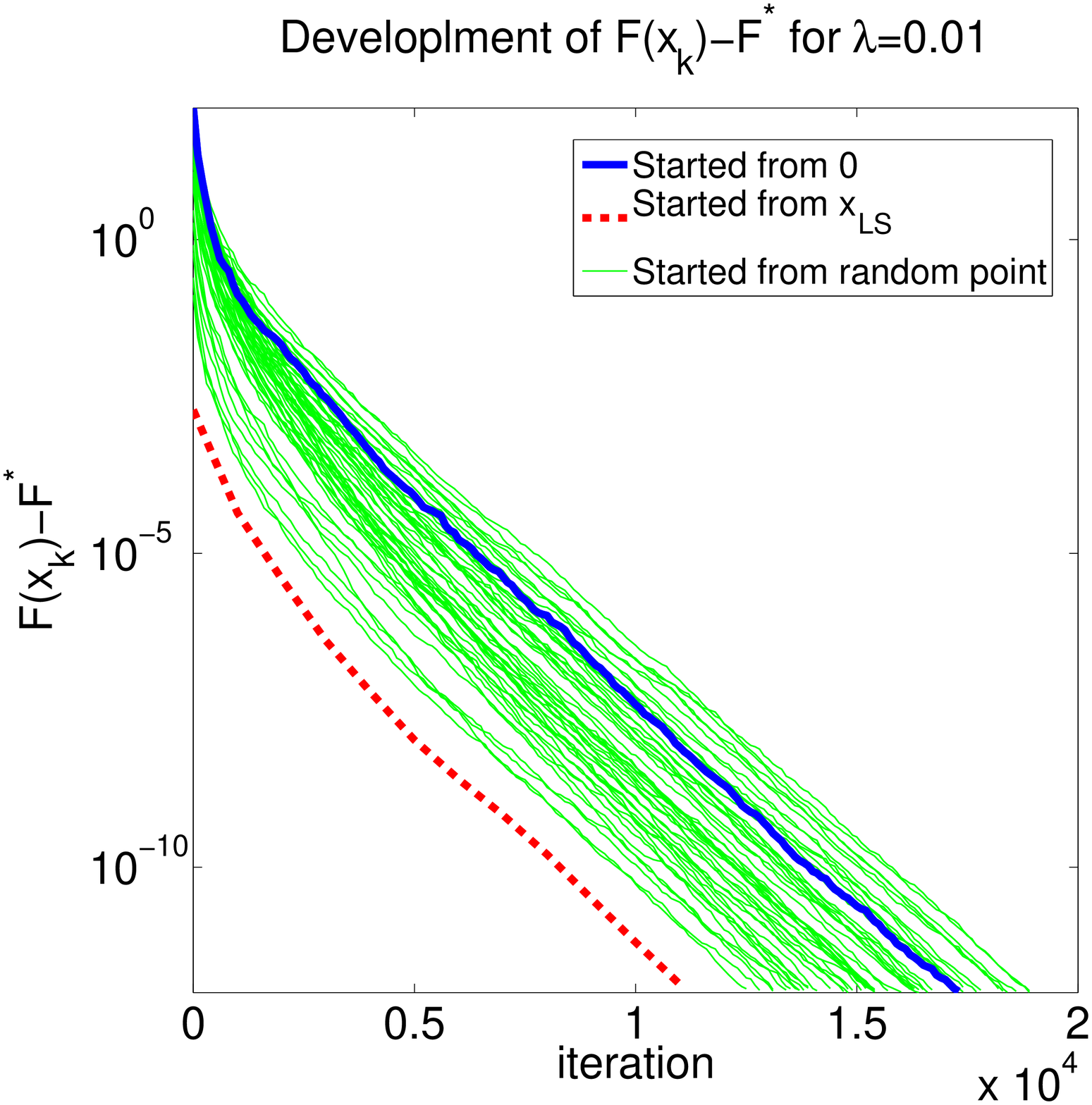}
 \hskip -1cm
 \includegraphics[width=7cm]{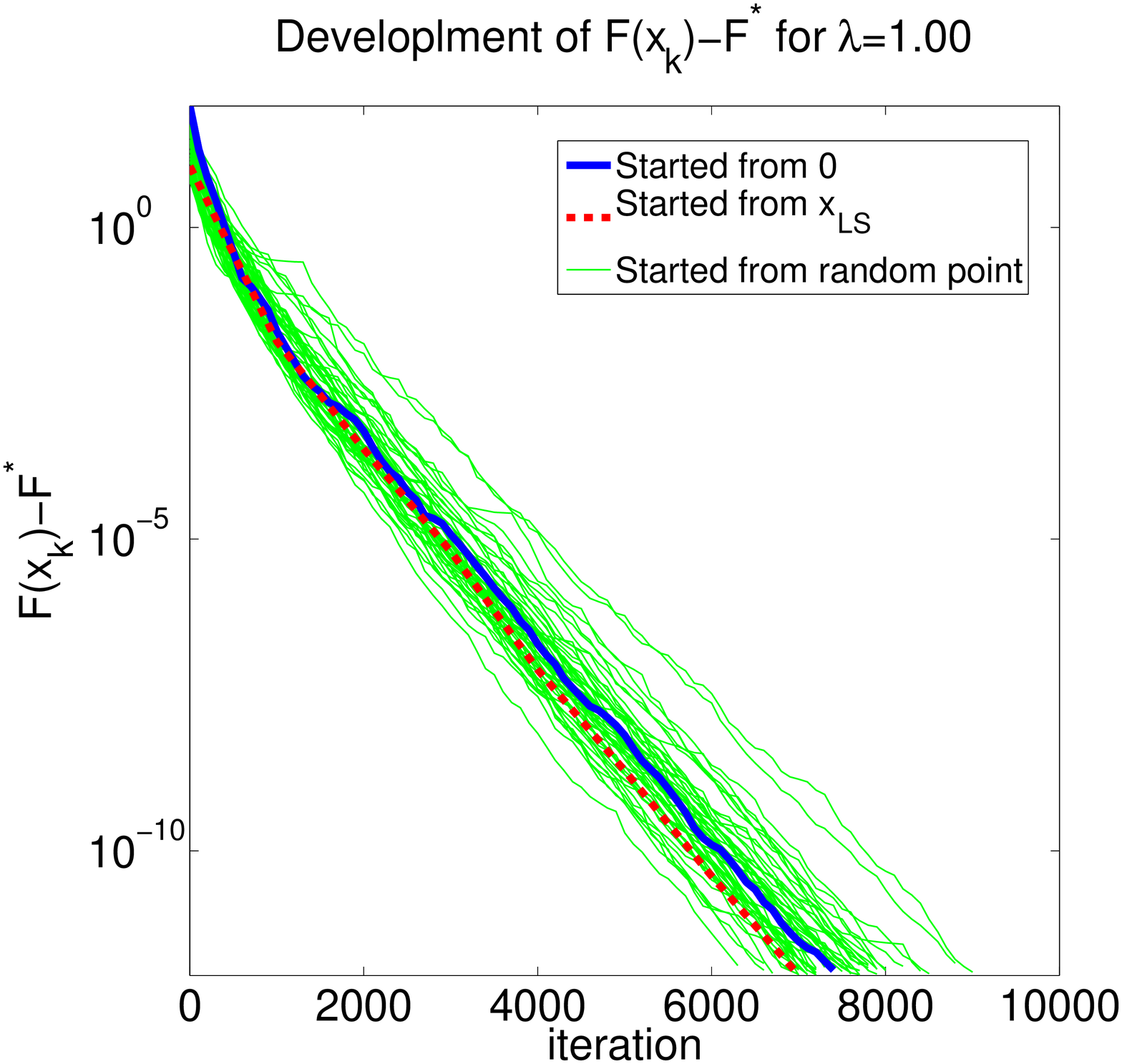}
 \caption{There does not seem to be any advantage in starting UCDC from a point on the line segment between $x_{LS}$ and the origin as opposed to starting it from the better of two endpoints.}
 \label{fig:warmstartwscomaprerandompoint}
\end{figure}


\subsection{Linear Support Vector Machines}\label{sec:exp:SVM}

Consider the problem of training a linear classifier with training examples $\{(x_1, y_1),\dots, (x_m,y_m)\}$,
where $x_i$ are the feature vectors and $y_i \in \{-1, +1\}$ the corresponding labels (classes). This problem is usually cast as an optimization problem of the form \eqref{eq:P},
\begin{equation}\label{eq:SVMproblemFormulation}
  \min_{w\in\R^n} F(w) = f(w) + \Psi(w),
\end{equation}
where
\[f(w) =  \gamma \sum_{i=1}^m \lf(w;x_i,y_i),\]
$\lf$ is a nonnegative convex loss function and $\Psi(\cdot)=\|\cdot\|_1$ for L1-regularized and $\Psi(\cdot)=\|\cdot\|_2$ for L2-regularized linear classifier. Some popular loss functions are listed in Table~\ref{tbl:lossFunctions}. For more details we refer the reader to \cite{Lin:2010:COMS} and the references therein; for a survey of recent advances in large-scale linear classification see \cite{Lin:IEEEsurvey}.

\begin{table}[!htp]
 \centering
 \begin{tabular}{c|c|c}
  $\lf(w;x_i,y_i)$ & name & property\\ \hline \hline
  $\max\{0, 1-y_j w^Tx_j \}$ & L1-SVM loss  (L1-SVM)& $C^0$ continuous\\
  $\max\{0, 1-y_j w^Tx_j \}^2$ & L2-SVM loss  (L2-SVM)& $C^1$ continuous\\
  $\log (1+e^{-y_j w^Tx_j})$ & logistic loss  (LG)& $C^2$ continuous  \\
\end{tabular}
\caption{A list of a few popular loss functions.}\label{tbl:lossFunctions}
\end{table}
Because our setup requires $f$ to be at least $C^1$ continuous, we will consider the L2-SVM and LG loss functions only. In the experiments below we consider the L1 regularized setup.

\nadpis{A few implementation remarks}
The Lipschitz  constants and coordinate derivatives of $f$ for the L2-SVM and LG loss functions are listed in Table~\ref{tbl:lipconstantsandparialderivatives}.

\begin{table}[!htp]
 \centering
 \begin{tabular}{c|c c}
 Loss function  & $\Lip_i$ & $\nabla_i f(w)$ \\
 \hline \hline
        &  & \\
 L2-SVM &   $\displaystyle 2\gamma \sum_{j=1}^m (y_j\vc{x_j}{i})^2$ & \qquad $\displaystyle -2\gamma \cdot \sum_{j\st -y_jw^Tx_j>-1} y_j \vc{x_j}{i}(1 -y_jw^Tx_j)$\\
        &  & \\
 LG     &   $\displaystyle \tfrac{\gamma}{4} \sum_{j=1}^m (y_j\vc{x_j}{i})^2$ & \qquad$\displaystyle -\gamma \cdot \sum_{j=1}^m y_j\vc{x_j}{i}\frac{ e^{-y_j w^Tx_j}}{1+e^{-y_j w^Tx_j}}$\\
        &  & \\
\end{tabular}
\caption{Lipschitz constants and coordinate derivatives for SVM.}\label{tbl:lipconstantsandparialderivatives}
\end{table}

For an efficient implementation of UCDC we need to be able to cheaply update the partial derivatives after each step of the method. If at step $k$  coordinate $i$ gets updated, via $w_{k+1} = w_k+ t e_i$, and we let $\vc{r_k}{j}  \eqdef -y_j w^Tx_j$ for $j=1,\dots,m$, then
\begin{equation}\label{eq:svm_update}\vc{r_{k+1}}{j} = \vc{r_{k}}{j} - t y_j  \vc{x_j}{i}, \quad j=1,\dots,m.\end{equation}
Let $o_i$ be the number of observations feature $i$ appears in, i.e., $ o_i = \#\{j:x_j^{(i)}\neq 0\}$. Then
 the update \eqref{eq:svm_update}, and consequently the update of the partial derivative (see Table~\ref{tbl:lipconstantsandparialderivatives}), requires $O(o_i)$ operations. In particular, in \emph{feature-sparse problems} where $\tfrac{1}{n}\sum_{i=1}^n o_i \ll m$, an average iteration of UCDC will be very cheap.

\subsubsection*{Small scale test} We perform only preliminary results on the dataset \verb"rcv1.binary"\footnote{\url{http://www.csie.ntu.edu.tw/~cjlin/libsvmtools/datasets/binary.html}}.
This dataset has 47,236 features and 20,242 training and 677,399 testing instances. We train the classifier on 90\% of training instances (18,217);
the rest we used for cross-validation for the selection of the parameter $\gamma$. In Table~\ref{tbl:parameterCSelection} we list cross-validation accuracy (CV-A) for various choices of $\gamma$ and testing accuracy (TA) on 677,399 instances. The best constant $\gamma$ is $1$ for both loss functions in cross-validation.

\begin{table}[!htp]
 \centering
 \begin{tabular}{c||r|r|r||r|r|r}
  Loss function & $\gamma$ & CV-A & TA& $\gamma$ & CV-A & TA\\
  \hline \hline
  L2-SVM & 0.0625 & 94.1\% & 93.2\% &  2 & 97.0\% & 95.6\%\\
	     & 0.1250 & 95.5\% & 94.5\% &  4 & 97.0\% & 95.4\%\\
	     & 0.2500 & 96.5\% & 95.4\% &  8 & 96.9\% & 95.1\%\\
	     & 0.5000 & 97.0\% & 95.8\% & 16 & 96.7\% & 95.0\%\\
	     & \textbf{1.0000} & \textbf{97.0\%} & \textbf{95.8\%} & 32 & 96.4\% & 94.9\%\\
\hline
LG	     & 0.5000 &  0.0\% &  0.0\% &  8 & 40.7\% & 37.0\% \\
         & \textbf{1.0000} & \textbf{96.4\%} & \textbf{95.2\%} & 16 & 37.7\% & 36.0\% \\
	     & 2.0000 & 43.2\% & 39.4\% & 32 & 37.6\% & 33.4\% \\
	     & 4.0000 & 39.3\% & 36.5\% & 64 & 36.9\% & 34.1\%  	
\end{tabular}
\caption{Cross validation accuracy (CV-A) and testing accuracy (TA) for various choices of $\gamma$.}\label{tbl:parameterCSelection}
\end{table}

In Figure~\ref{fig:dependenceTAonNumberOfIterations} we present dependence of TA on the number of iterations we run UCDC for (we measure this number in multiples of $n$). As you can observe, UCDC finds good solution after $10\times n$ iterations, which for this data means less then half a second. Let us remark that we did not include bias term or any scaling of the data.

\begin{figure}[!htp]
 \centering
 \includegraphics[width=14cm]{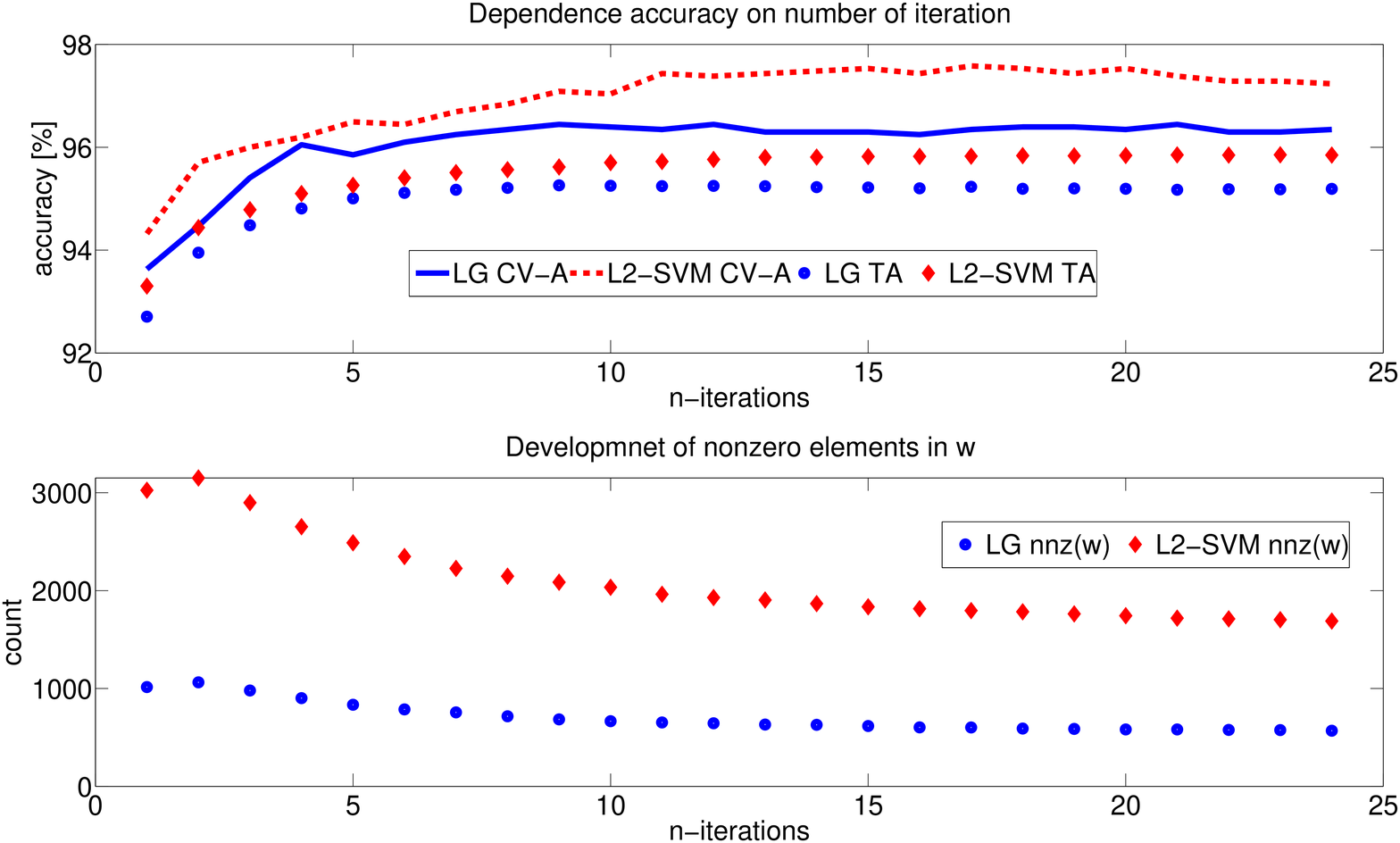}
 \caption{Dependence of tested accuracy (TA) on the number of full passes through the coordinates.}
 \label{fig:dependenceTAonNumberOfIterations}
\end{figure}

\nadpis{Large scale test} We have used the dataset \verb"kdd2010" (bridge to algebra)\footnote{\url{http://www.csie.ntu.edu.tw/~cjlin/libsvmtools/datasets/binary.html}}, which has
29,890,095 features and 19,264,097 training and 748,401 testing instances. Training the classifier on the entire  training set required approximately 70 seconds in the case of L2-SVM loss and 112 seconds in the case of LG loss. We have run UCDC for $n$ iterations. 

\bibliographystyle{plain} 
\bibliography{literature}

\end{document}